\documentclass[10pt]{amsart}  
\usepackage{amsmath,amssymb,amsfonts,times,hyperref}
\usepackage{graphicx }

\input{xy}
\xyoption{all}

\title{The theta number of simplicial complexes}

\author{Christine Bachoc} 
\address{Institut de Math\'ematiques de Bordeaux, UMR 5251, Universit\'e de Bordeaux, 351 Cours de la Lib\'eration, 33400 Talence, France.}
\email{christine.bachoc@u-bordeaux.fr}
\author{Anna Gundert} 
\address{Mathematisches Institut, Universit\"at zu K\"oln, Weyertal 86-90, 50931 K\"oln, Germany.}
\email{anna.gundert@uni-koeln.de}
\author{Alberto Passuello} 
\address{Institut de Math\'ematiques de Bordeaux, UMR 5251, Universit\'e de Bordeaux, 351 Cours de la Lib\'eration, 33400 Talence, France.}
\email{alberto.passuello@u-bordeaux.fr}

\date{\today}
\newtheorem{defi}{Definition}[section]
\newtheorem{definition}[defi]{Definition}
\newtheorem{proposition}[defi]{Proposition}
\newtheorem{theorem}[defi]{Theorem}
\newtheorem{remark}[defi]{Remark}

\newtheorem{lemma}[defi]{Lemma}
\newtheorem{example}[defi]{Example}

\newcommand{\R}{{\mathbb{R}}}

\newcommand{\la}{\langle}
\newcommand{\ra}{\rangle}
\newcommand{\one}{{\bf 1}}

\newcommand{\Id}{{\bf I}}
\newcommand{\CC}{\mathcal{C}}
\newcommand{\HH}{\mathcal{H}}

\newcommand{\YY}{\mathcal{Y}}
\newcommand{\OX}{\overline{X}}

\newcommand{\ima}{\operatorname{im}}
\newcommand{\trace}{\operatorname{trace}}
\newcommand{\Ind}{\operatorname{Ind}}
\newcommand{\diag}{\operatorname{diag}}

\newcommand{\Lu}{L^{\uparrow}}
\newcommand{\Ld}{L^{\downarrow}}
\newcommand{\lk}{\operatorname{lk}}

\begin{document}

\begin{abstract}
We introduce a generalization of the celebrated Lov\'asz theta number
of a graph to simplicial complexes of arbitrary dimension.
Our generalization takes advantage of real simplicial cohomology theory, in particular combinatorial Laplacians, 
and provides a semidefinite programming upper bound of the independence number of a simplicial complex.
We consider properties of the graph theta number such as the relationship to Hoffman's ratio bound and to the chromatic number and study how they extend to higher dimensions. 
 Like in the case of graphs, the higher dimensional theta number can be extended to a hierarchy
of semidefinite programming upper bounds reaching the independence number. We analyse
the value of the theta number and of the hierarchy for dense random simplicial complexes.
\end{abstract}

\maketitle

\section{Introduction}\label{sec:introduction}

The theta number $\vartheta(G)$ of a graph $G$ was  introduced by L. Lov\'asz in his seminal paper
\cite{Lov}, in order to provide
spectral bounds of the independence number and of the chromatic number
of $G$.
 In modern terms, $\vartheta(G)$ is the optimal value of
a semidefinite program, and as such is computationally easy;
in contrast, the independence number $\alpha(G)$ and the chromatic
number $\chi(G)$ are difficult to compute. These graph invariants satisfy
the following inequalities, where $\overline{G}$ denotes the
complement of $G$:
\begin{equation}
\alpha(G)\leq \vartheta(G)\leq \chi(\overline{G}).
\end{equation}

The inequality $\alpha(G)\leq \vartheta(G)$ was one of the main ingredients in
Lov\'asz' proof  of the Shannon conjecture on the capacity of the
pentagon \cite{Lov}. More generally, this inequality plays a central role in
extremal combinatorics, sometimes in a disguised form: to cite a few, the Delsarte linear
programming method in coding theory \cite{De} and recent generalizations of
Erd\"os-Ko-Rado theorems \cite{DLV, EFP, EFF} can be interpreted as  instances of this inequality. Analogs of the theta
number in geometric settings have lead 
to many advances in packing problems (see \cite{OV} and references
therein), in particular the very
recent solutions to the sphere packing problems in dimensions $8$ and
$24$ \cite{CKMRV,Vi}. 

Our aim in this paper is to generalize this graph parameter to higher
dimensions, in the framework of \emph{simplicial
  complexes}. Let us recall that an (abstract)
  simplicial complex $X$ on a finite set $V$ is a family of subsets of $V$ called
 \emph{faces} that is closed under taking subsets. We refer
to Section 1 for basic
definitions and results about simplicial complexes. Graphs fit in this
framework, being simplicial complexes of dimension $1$.
 In recent years, considerable work has been
devoted to generalizing  the classical theory of graphs to
this higher-dimensional setting. Much of the efforts have focused on the notion of expansion (see, e.g., \cite{DKW, EK, Gr, KKL, Lu, PRT}), but other natural 
concepts such as random walks \cite{PR}, trees \cite{DKM,Ka}, planarity \cite{MTW}, girth \cite{DGK, LuMe}, independence and chromatic numbers \cite{EGL, Go} have been extended to higher dimensions. Some of these notions were introduced and studied previously in the context of \emph{hypergraphs}. Pure $k$-dimensional
simplicial complexes are essentially $(k+1)$-uniform
hypergraphs, but the topological point of view brings the machinery of algebraic topology such as homology theory to the subject.


The familiar graph-theoretic notions
of independence number and of chromatic number extend in a natural
way to this setting: For a $k$-dimensional simplicial complex $X$, an independent set is a set of vertices that does not contain any
maximal face of $X$, and  the \emph{independence number} $\alpha(X)$ is the maximal cardinality of
an independent set. The \emph{chromatic number}\footnote{In the study of hypergraphs, the chromatic number $\chi(X)$ is also known as the \emph{weak chromatic number} while $\chi(X_1)$, the chromatic number of the $1$-skeleton, is known as the \emph{strong chromatic number}.} $\chi(X)$ is the least number of colors needed to color the vertices so that no
maximal face of $X$ is monochromatic, in other words, it is the
smallest number of parts of a partition of the vertices into independent
sets.

In order to define the theta number $\vartheta_k(X)$ of a pure
$k$-dimensional simplicial complex $X$, we will follow an approach that
leads in a natural way  to the inequality $\alpha(X)\leq
\vartheta_k(X)$. The main
idea is to associate to an independent set $S$ a certain
matrix, and then to design a semidefinite
program that captures as many properties of this
matrix as possible. The matrix that we associate to an independent set is (up to
a multiplicative factor) a
submatrix of the \emph{down-Laplacian of the complete complex}. 
In the case of dimension $1$, the down-Laplacian is simply the all-one
matrix, and we end up with one of the many formulations of the Lov\'asz theta
number.

Our first task will be to compare $\vartheta_k(X)$ to the eigenvalue upper bound
of $\alpha(X)$ proved by Golubev in \cite{Go}. This upper bound involves
for $0\leq i\leq k-1$, the largest eigenvalues $\mu_i$ of the $i$-th up-Laplacians of $X$ and the minimal degrees $d_i$ of the $i$-faces
of $X$:
\begin{equation}\label{eq:bound Golubev}
\alpha(X)\leq n\bigg(1-\frac{(d_0+1)(d_1+2)\dots
  (d_{k-2}+k-1)d_{k-1}}{\mu_0\dots \mu_{k-1}}\bigg).
\end{equation}
When every possible $(k-1)$-face is contained in at least one $k$-face,
i.e., when $X$ has a \emph{complete $(k-1)$-skeleton}, this inequality simplifies
to 
\begin{equation}\label{eq:bound Golubev complete}
\alpha(X)\leq n\bigg(1-\frac{d_{k-1}}{\mu_{k-1}}\bigg)
\end{equation}
and can thus be seen as a natural generalization of the celebrated
\emph{ratio bound} for graphs attributed to Hoffman (see, e.g., \cite[Theorem
3.5.2]{BH}). In that case, we will show that
\begin{equation*}
\vartheta_k(X)\leq n\bigg(1-\frac{d_{k-1}}{\mu_{k-1}}\bigg),
\end{equation*}
therefore $\vartheta_k(X)$ provides an upper bound of $\alpha(X)$ that is at
least as good as \eqref{eq:bound Golubev complete}. In the case of a non-complete $(k-1)$-skeleton, Golubev's bound and $\vartheta_k(X)$ turn out to
be incomparable, as we will see in examples below.

The theta number of a graph has many very nice properties; some
of them, although unfortunately not all of them, can be generalized to higher
dimensions. Most of this paper is devoted to determining which
of the properties of the graph theta number extend to our notion of the
theta number of simplicial complexes.

The relationship to the chromatic number generalizes only partially. Indeed, the inequality $\alpha(X)\leq
\vartheta_k(X)$ immediately leads to the inequality
$n/\vartheta_k(X)\leq \chi(X)$. However, in the case of graphs, the stronger
inequality $\vartheta(\overline{G})\leq \chi(G)$ holds. We will see
that its  natural analog in the setting of $k$-complexes would be that 
$\vartheta_k(\overline{X})\leq k\chi(X)$ and that this inequality does not hold in
general. Instead, we will introduce an ad hoc notion of chromatic
number for simplicial complexes, denoted $\chi_k(X)$, and show that the inequality $\vartheta_k(\overline{X})\leq
\chi_k(X)$ holds. 
While $\chi(X)$ is defined using vertex colorings, the definition of
$\chi_k(X)$ is based on colorings of $(k-1)$-faces respecting
orientations. Moreover, it is tightly related to a notion of
\emph{homomorphisms} between pure $k$-dimensional simplicial complexes that we introduce
and that may be of interest by itself.

A very interesting benefit of the theta number of a graph is that it
is possible to expand it  into \emph{hierarchies} of semidefinite upper bounds of
the independence number; Lassere's hierarchy based on polynomial
optimization principles is one of the most popular (see \cite{Las, Lau}). We will see that
a similar situation holds in higher dimensions: to a
pure $k$-dimensional complex $X$ we will associate a sequence $\hat{\vartheta}_{\ell}(X)$ for
$\ell=k,\dots,\alpha(X)$ such that 
\[
\alpha(X) = \hat{\vartheta}_{\alpha(X)}(X)\leq
\dots \leq \hat{\vartheta}_{\ell}(X) \leq \dots\leq
\hat{\vartheta}_{k}(X)\leq \vartheta_k(X).
\]
In order to define
$\hat{\vartheta}_{\ell}(X)$, we will proceed in two steps: in a first step, we
define a natural sequence
$\vartheta_{\ell}(X)$ for $\ell=k,k+1,\dots,\alpha(X)$; in a second
step, we modify the definition of $\vartheta_{\ell}(X)$ slightly in such a way that the sequence of its values  decreases.

Our last results concern the theta number of random simplicial complexes $X^k(n,p)$ from the model proposed by Linial and Meshulam in \cite{LiMe}. This model is a higher-di\-men\-sional analog of the Erd\H{o}s-R\'enyi model $G(n,p)$ for random graphs and has gained increasing attention in recent years (see \cite{Kah} for a survey).

We show that $\vartheta_k(X^k(n,p))$ is of the order of
$\sqrt{(n-k)(1-p)/p}$ for probabilities $p$ such that $c_0\log(n)/n\leq p\leq
1-c_0\log(n)/n$ for some constant $c_0$. This result extends the known estimates for the
value of the theta number of the random graph $G(n,p)$.

The paper is organized as follows: Sections 2 and 3 recall basic definitions
and properties of simplicial complexes and semidefinite
programming. Section 4 recalls properties of the theta number of a
graph that serve as a guideline for the theta number  of a
$k$-dimensional simplicial complex, which is introduced in Section
5. Section 6 computes the theta number of certain basic
families of $2$-dimensional simplicial complexes. Section 7 discusses
chromatic numbers and Section 8 the hierarchy
of theta numbers. The final Section 9 contains the analysis of the theta number of random simplicial complexes.

\section{Simplicial complexes}\label{sec:complex}

Let $V=\{v_1,\dots, v_n\}$ be a finite set. We will use the notation
$\binom{V}{k}$ for the set of $k$-subsets of $V$. Let us recall that an \emph{(abstract)
  simplicial complex} $X$ on a vertex set $V$ is a family of subsets of $V$ (called
the \emph{faces} of $X$), such that if $F\in X$, then all subsets of $F$
also belong to $X$. The \emph{dimension} of a face $F\in X$ is $|F|-1$, and we denote by $X_i$ the set of $i$-dimensional faces of
$X$, with the convention $X_{-1}=\{\emptyset\}$.   Note that we do not require every element in $V$ to be a $0$-face of $X$, so $X_0$ can be a proper subset of $V$.
The \emph{$i$-skeleton} of $X$ is the simplicial complex $X_{-1}\cup
X_0\cup\dots \cup X_i$. 

A simplicial complex $X$ is said to be of dimension $k\geq 0$, if $k$ is
the maximal dimension of any of its faces. For example, a graph is a simplicial complex of dimension $1$. Going back to
the general case, if $X$ is of dimension $k$, and if moreover \emph{all}
maximal (with respect to inclusion) faces of $X$ are of dimension $k$, then
$X$ is said to be \emph{pure}. Unless explicitly mentioned, we will
only consider pure complexes.

A basic example of a pure $k$-dimensional simplicial complex is the
\emph{complete $k$-complex} $K_n^k$,
whose faces are all the subsets of $[n]=\{1,\dots,n\}$ 
that have at most $(k+1)$ elements. 

We note that in order to define a pure simplicial complex
of dimension $k$, it is enough to specify its set of $k$-dimensional
faces. In particular, the \emph{complementary complex} $\overline{X}$
of a  pure simplicial complex
of dimension $k$, is again a pure simplicial complex of dimension $k$, whose $k$-dimensional faces are 
those $(k+1)$-subsets of $V$ that do not belong to $X_k$ (we
adopt the convention that the empty complex, whose set of faces is empty, is pure
of dimension $k$ for all $k\geq 0$).

Let $X$ be a simplicial complex; we assume that every face of $X$ is
endowed with an \emph{orientation}, i.e., a local ordering of its
vertices.  Then, if $F\in X_i$ and $K\in X_{i-1}$, an \emph{oriented
  incidence number} $[F:K]\in\{0,\pm 1\}$ can be defined. 
Often, the
orientation of the faces is induced by a global ordering of the vertex
set $V$; in that case, if  $F=\{x_0,x_1,\dots,x_i\}$ where $x_0<x_1<\dots<x_i$ with respect to
this ordering,
\begin{equation*}
[F:K] =\left\{
\begin{array}{ll}
(-1)^j & \text{if }K\subset F \text{ and } F\setminus K=\{x_j\},\\
0 &\text{otherwise.}
\end{array}
\right.
\end{equation*}

The vector space of functions from $X_i$ to $\R$ is denoted by $\CC^i(X;\R)$ and its elements are called \emph{$i$-dimensional cochains of $X$ with coefficients in $\R$}. 
The \emph{coboundary map} 
$\delta_i: \CC^i(X;\R)\to \CC^{i+1}(X;\R)$ is defined for $-1\leq i<\dim(X)$ by
\begin{equation*}
(\delta_if)(H)=\sum_{F\in X_i} [H:F] f(F).
\end{equation*}
The image of $\delta_{i-1}$ is the subspace $B^i(X;\R)$ of \emph{$i$-dimensional coboundaries}, and the 
kernel  of $\delta_{i}$ is the subspace $Z^i(X;\R)$ of \emph{$i$-dimensional cocycles}. Because the coboundary maps satisfy 
$\delta_i\circ \delta_{i-1}=0$, we have $B^i(X;\R)\subseteq Z^i(X;\R)$. The quotient group 
\begin{equation*}
{H}^i(X;\R):=Z^i(X;\R)/B^i(X;\R).
\end{equation*}
is then called the \emph{$i$-th cohomology group of $X$ with coefficients in $\R$}.

Analogously, we can define the homology groups of a simplicial complex. For this, the spaces $\CC^i(X;\R)$ are endowed with the standard inner product 
$\la f,g\ra=\sum_{F\in X_i} f(F)g(F)$ and the \emph{boundary map} $\partial_{i+1}=\delta_i^*: \CC^{i+1}(X;\R)\to \CC^{i}(X;\R)$ is defined as the adjoint of the coboundary map $\delta_{i}$.
We have, for $F\in X_i$,
\begin{equation*}
(\partial_{i+1}f)(F)= \sum_{H\in X_{i+1}} [H:F] f(H).
\end{equation*}
The spaces of \emph{boundaries} $B_i(X;\R):=\ima \partial_{i+1}$ and of \emph{cycles} $Z_i(X;\R):=\ker \partial_i$ are subspaces of $\CC^i(X;\R)$ satisfying
$B_i(X;\R)\subseteq Z_i(X;\R)$ and thus define the 
\emph{$i$-th reduced homology group} of $X$ 
\begin{equation*}
{H}_i(X;\R):=Z_i(X;\R)/B_i(X;\R).
\end{equation*}
Moreover, by duality we have that $Z_i(X;\R)=B^i(X;\R)^\perp$ and
$Z^i(X;\R)=B_i(X;\R)^\perp$.
The following diagram  summarizes these linear maps for $0\leq i\leq \dim(X)-1$:
\begin{equation*}
\xymatrix{
\CC^{i+1}(X;\R)\ar@<-2.5pt>[r]_{\partial_{i+1}} & \CC^{i}(X;\R)\ar@<-2.5pt>[l]_{\delta_i}
\ar@<-2.5pt>[r]_{\partial_i} &\CC^{i-1}(X;\R)\ar@<-2.5pt>[l]_{\delta_{i-1}}
}
\end{equation*}

The $i$-th \emph{up-Laplacian} $\Lu_i$ and $i$-th
\emph{down-Laplacian} $\Ld_i$ of $X$ are the following self-adjoint and positive
semidefinite operators on $\CC^i(X;\R)$:
\begin{equation*}
\Ld_i:=\delta_{i-1}\partial_i, \quad \Lu_i:=\partial_{i+1}\delta_i.
\end{equation*}
By definition, $\Lu_i\Ld_i=\Ld_i\Lu_i=0$.  Furthermore, it is not hard to see that $\ker \Ld_i=Z_i(X;\R)$, 
$\ima \Ld_i=B^i(X;\R)$, $\ker \Lu_i=Z^i(X;\R)$, and $\ima \Lu_i=B_i(X;\R)$. For
\begin{equation*}
\HH_i(X;\R):=Z_i(X;\R)\cap Z^i(X;\R),
\end{equation*}
 we have the \emph{Hodge decomposition} of $\CC^i(X;\R)$ into pairwise orthogonal subspaces
\begin{equation*}
\CC^i(X;\R)=\HH_i(X;\R)\oplus B^i(X;\R)\oplus B_i(X;\R).
\end{equation*}
In particular, $\HH_i(X;\R)\simeq H^i(X;\R)\simeq H_i(X;\R)$.

The characteristic functions $e_F$ of faces $F\in X_i$ are called \emph{elementary cochains}; they form an orthonormal basis of $\CC^i(X;\R)$.
In order to express the 
 matrices of the Laplacian operators in this basis we introduce the following notation:
for $F\in X_i$, let $\deg(F)$ denote the \emph{degree} of $F$,
i.e., the number of $(i+1)$-faces of $X$ that contain $F$.  For
$(F,F')\in X_i^2$, such that $|F\cap F'|=i$, let
\begin{equation*}
\epsilon_{F,F'}:=[F:F\cap F'][F':F\cap F'].
\end{equation*}
We note that, if $F\cup F'\in X_{i+1}$, we can express $\epsilon_{F,F'}$ also as
\begin{equation*}
\epsilon_{F,F'}=-[F\cup F':F][F\cup F':F']. 
\end{equation*}
For $(F,F')\in X_i^2$, such that $|F\cap F'|\neq i$, we set $\epsilon_{F,F'}=0$.
Then, it is easy to see that
\begin{equation*}
(\Ld_i)_{F,F'}=\left\{ 
\begin{array}{ll}
i+1 & \text{ if } F=F'\\
\epsilon_{F,F'} & \text{ otherwise}
\end{array}
\right.
\end{equation*}
and 
\begin{equation*}
(\Lu_i)_{F,F'}=\left\{ 
\begin{array}{ll}
\deg(F) & \text{ if } F=F'\\
-\epsilon_{F,F'} &\text{ if } F\cup F'\in X_{i+1}\\
0 & \text{ otherwise}
\end{array}
\right.
\end{equation*}
where we use the same notations for the operators and for their matrices in the basis of elementary cochains.

\begin{example} In the case of the simplicial complex associated to a graph $G=(V,E)$, defined by
$X_{-1}=\{\emptyset\}$, $X_0=V$ and $X_1=E$, we find that $\Ld_0=J$ is
the all-ones matrix and 
$\Lu_0$ is equal to the combinatorial Laplacian  $L=D-A$ where $D$ is the diagonal matrix with the degrees of the vertices as diagonal elements and $A$ is the adjacency matrix of 
the graph.
\end{example}

\begin{example}\label{ex:complete} For the complete $k$-complex $K_n^k$, and for $0\leq
  i\leq k-1$, it is easy to verify that
\begin{equation*}
\Lu_i+\Ld_i=nI.
\end{equation*}
Together with the property $\Lu_i\Ld_i=0$, we obtain that $(\Lu_i)^2=n\Lu_i$ and that
$(\Ld_i)^2=n\Ld_i$. So $n$ is the only non zero eigenvalue of the up
and down Laplacians. Computing the traces of these operators gives the multiplicities of this eigenvalue, namely 
$\binom{n-1}{i}$ for $\Ld_i$ and
$\binom{n-1}{i+1}$ for $\Lu_i$. So we have
\begin{equation*}
\ker(\Lu_i-nI)=\ima(\Lu_i) =B_i, \quad \dim(B_i)=\binom{n-1}{i+1}, 
\end{equation*}
\begin{equation*}
\ker(\Ld_i-nI)=\ima(\Ld_i)=B^i, \quad \dim(B^i)=\binom{n-1}{i},
\end{equation*}
and, as these
dimensions add up to $\binom{n}{i+1}=\dim(C^i)$,  $\HH_i=\{0\}$.
\end{example}

We conclude this section  by recalling the definition of the adjacency matrix of 
 a $k$-dimensional simplicial complex $X$: it is the matrix $A$ such that $\Lu_{k-1}=D-A$ 
where  $D$ is the diagonal matrix encoding the degrees 
of the $(k-1)$-faces. In other words,
\begin{equation*}
A_{F,F'}=\left\{ 
\begin{array}{ll}
\epsilon_{F,F'} &\text{ if } F\cup F'\in X_{k}\\
0 & \text{ otherwise}
\end{array}
\right.
\end{equation*}
We note that in dimension $1$ this definition coincides with the usual notion of the adjacency matrix of a graph.

\section{Semidefinite programming}\label{sec:sdp}

In this section, we gather basic facts about  semidefinite
programs. For further information we refer to standard
references such as \cite{BN}, \cite{BV}, \cite{VB}.

Semidefinite programs (SDP for short) are special cases of convex optimization programs that admit
efficient algorithms, such as algorithms based on the so-called
interior point method. They generalize linear programs
and have turned out to be very useful for providing polynomial time
approximations of hard problems in many areas, especially in 
combinatorics (see, e.g., \cite{GM} and \cite[Chapter 6]{AL}). 

For a matrix $A\in \R^{n\times n}$ we say that $A$ is positive semidefinite, denoted by $A\succeq 0$,  if $A$ is
real-valued, symmetric, and if all its eigenvalues are nonnegative. If moreover none of its eigenvalues are equal to zero, $A$ is positive definite ($A\succ 0$).  The set of all positive
semidefinite matrices is a cone denoted by
$\R^{n\times n}_{\succeq 0}$. The space of real symmetric matrices is endowed with the standard inner product 
$\langle A,B\rangle=\trace(AB)$. 

Given $(c_1,\dots,c_m)\in \R^m$ and symmetric
matrices $A_0,\dots,A_m$ of size $n$, the following optimization problem is
a \emph{semidefinite program in primal form}:
\begin{equation*}
p^*=\sup \{ \langle A_0, Z\rangle \ :\ Z\in \R^{n\times n}_{\succeq 0}
,\ \langle A_i, Z\rangle =c_i\text{ for all }1\leq i\leq m\}.
\end{equation*}

In other words, this program asks for the supremum of a linear form, where this
supremum is taken over
the intersection of the cone of positive semidefinite matrices
with an affine space.

A \emph{feasible solution} of this program is a matrix $Z$ that
satisfies the required constraints: $Z\in \R^{n\times n}_{\succeq 0}$
and $\langle A_i, Z\rangle =c_i$. It is an \emph{optimal solution} if
its \emph{objective value} $\langle A_0,Z\rangle$ is equal to
$p^*$. If there is no feasible solution, we let $p^*=-\infty$. 

The following \emph{dual program} is attached to the primal program:
\begin{equation*}
d^*=\inf \{ c_1 x_1+\dots +c_m x_m \ :\ (x_1,\dots,x_m)\in \R^m,\ -A_0+ x_1A_1+\dots+x_mA_m
\succeq 0 \}.
\end{equation*}

The terms 'primal' and 'dual' do not refer to a specific class of
programs: Despite their apparent difference, any of these programs can be put in
the form of the other, and, as expected, dualizing twice returns the initial program.

The inequality $p^*\leq d^*$, referred to as \emph{weak duality},
always holds, and under some mild conditions even \emph{strong duality},
i.e., $p^*=d^*$, holds. Strong duality is guaranteed
if the SDP satisfies the so-called \emph{Slater's conditions}, of which we will use the following version: If an SDP has a 
strictly feasible primal solution, i.e., if there is a feasible solution $Z$ of the primal program such that $Z\succ 0$,  and a strictly feasible dual solution, i.e., there exists $(x_1,\dots,x_m)$ such that $-A_0+ x_1A_1+\dots+x_mA_m \succ 0$, then 
strong duality holds and, moreover, there are optimal solutions for both the primal and the dual program.

\section{The theta number of a graph}\label{sec:theta graph}

In this section,  we introduce the theta number of a graph
$G=(V,E)$.
Our presentation will serve as a guideline for the generalization to higher
dimensional simplicial complexes.

Let $S$ be an independent set of
$G$, i.e., a subset of $V$ not containing any edges.
The set
$S$ naturally defines a vector $\one_S\in \R^V$, namely its characteristic
vector. We consider the matrix $Y^S:=\one_S\one_S^T$, whose entries are given by:
\begin{equation*}
Y^S_{v,v'} = \left\{\begin{array}{ll}
0 &\text { if } \{v,v'\} \nsubseteq S\\
1  &\text{ otherwise}.
\end{array}\right.
\end{equation*}
The following properties of $Y^S$ motivate the definition of
$\vartheta(G)$: $Y^S$ is a positive semidefinite
matrix such that $Y^S_{v,v'}=0$ if $\{v,v'\}\in E$.
Furthermore, the
cardinality of $S$ can be recovered in two different ways from
$Y^S$: If $I$ and $J$ stand as
usual for the identity matrix and the all-ones matrix,  we have $\langle
I,Y^S\rangle=|S|$ and $\langle J,Y^S\rangle=|S|^2$. So, if we set
\begin{equation}\label{eq:theta primal}
\vartheta(G)=\sup\{ \langle J,Y\rangle \ :\  Y\in \R^{V\times V}, \ Y\succeq 0,\  \langle I,Y\rangle
=1,\  Y_{v,v'}=0 \ \text{ if }\{v,v'\}\in E \}
\end{equation}
 the matrix $|S|^{-1} Y^S$ is feasible for \eqref{eq:theta primal}  and we
get that $|S|\leq \vartheta(G)$.

Because \eqref{eq:theta primal} is a semidefinite program, its optimal value $\vartheta(G)$ 
can be approximated numerically up to
arbitrary precision in polynomial time in the size of $G$. 
If, instead of a sharp numerical value, one aims for a rougher upper bound of
$\vartheta(G)$, the
\emph{dual formulation} of \eqref{eq:theta primal}
is often more convenient:
\begin{equation}\label{eq:theta dual}
\vartheta(G)=\inf \{\lambda_{\max}(Z) \ :\ Z\in \R^{V\times V},\ Z=J+T,\ T_{v,v'}=0 \
\text{ if }\{v,v'\}\notin E \}.
\end{equation}
Here, $\lambda_{\max}(Z)$
denotes the largest eigenvalue of $Z$.

To illustrate this principle we consider a classical example. 
For any matrix $T$ such that $T_{v,v'}=0$ for all
$\{v,v'\}\notin E$, the dual formulation of
$\vartheta(G)$ provides the inequality $\alpha(G)\leq
\lambda_{\max}(J+T)$.
A possible choice for $T$ is a multiple of
the adjacency matrix $A$ of $G$, say $T=tA$. The best bound is obtained for $t$
minimizing $\lambda_{\max}(J+tA)$.
For $d$-regular graphs, the matrices $J$ and $A$ commute, so the eigenvalues of
$J+tA$ are easy to analyze. The optimal choice of
$t$ then leads to the so-called \emph{ratio bound} attributed to Hoffman (see, e.g., \cite[Theorem
3.5.2]{BH}):

\begin{equation}\label{eq:ratio bound}
\alpha(G)\leq \frac{-|V|\lambda_{\min}(A)}{d-\lambda_{\min}(A)}.
\end{equation}

\section{The  theta number of a simplicial complex}\label{sec:thetak}

We now move to higher dimensions and define the theta number of a
$k$-dimensional simplicial complex $X$.
As suggested in the introduction, the down-Laplacian
$\Ld_{k-1}$ of the complete complex $K_n^k$ will play the role of the all-ones
matrix $J$ in \eqref{eq:theta primal} and \eqref{eq:theta dual}. 
Recall that $\Ld_{k-1}$ is the matrix indexed by $\binom{V}{k}$
that is defined by:

\begin{equation*}
(\Ld_{k-1})_{F,F'}=\left\{ 
\begin{array}{ll}
k & \text{ if } F=F'\\
\epsilon_{F,F'} & \text{ otherwise}
\end{array}
\right.
\end{equation*}

We note that this matrix may not be the down-Laplacian of the complex $X$. Obviously, this is the case if and only if $X$ has a complete
$(k-1)$-skeleton, otherwise the down-Laplacian of $X$ is a principal
submatrix of $\Ld_{k-1}$. From now on, to avoid confusion, we will
denote the matrices associated to $X$ by $\Ld_i(X)$, $\Lu_i(X)$ and
reserve the notations $\Ld_i$, $\Lu_i$ for the complete complex.

Let $S\subset V$ be an independent set of $X$. 
Following the same strategy as in the case of graphs,
 we consider the following matrix $Y^S$, indexed by $\binom{V}{k}$:
\begin{equation}\label{e26}
(Y^S)_{F,F'} = \left\{\begin{array}{ll}
0 &\text { if } F \cup F' \nsubseteq S\\
(\Ld_{k-1})_{F,F'} &\text{ otherwise}.
\end{array}
\right.
\end{equation}
We have $Y^S=\delta_{\binom{S}{k}}\delta_{\binom{S}{k}}^T$, where as a generalization of the characteristic vector of $S$, we consider the matrix $\delta_{\binom{S}{k}}$ defined as follows:
\[
\big(\delta_{\binom{S}{k}}\big)_{K,F} =  \left\{\begin{array}{ll}
0 &\text { if } F\nsubseteq S\\
(\delta_S)_{K,F}  &\text{ otherwise},
\end{array}
\right.
\]
where $K \in \binom{V}{k-1}$, $F \in \binom{V}{k}$ and $\delta$ is the matrix of the boundary operator $\delta_{k-2}$ with respect to the basis of elementary cochains.
The properties of $Y^S$ lead to the following definition of
$\vartheta_k(X)$:

\begin{definition}\label{def:thetak}
Let $X$ be a pure $k$-dimensional complex on $V$,
and let $\Ld_{k-1}$ be the down Laplacian of the complete complex on $V$. Let:
\begin{equation}\label{eq:thetak primal}
\begin{array}{rl}
\vartheta_{k}(X):=\sup\big\{ \langle \Ld_{k-1},Y\rangle \ :&Y\in
                                                             \R^{\binom{V}{k}\times
                                                             \binom{V}{k}},\
                                                             Y\succeq
0, \ \langle I,Y\rangle=1, \\
&Y_{F,F'}=0 \text{ if } F\cup F' \in
  X_k,\\
&Y_{F,F'}=0 \text{ if }|F\cup F'|\geq k+2,\\
&\epsilon_{F,F'} Y_{F,F'}=\epsilon_{F'',F^\dag}Y_{F'',F^\dag} \text{ if } F\cup
  F'=F''\cup F^\dag\big\}
\end{array}
\end{equation}
\end{definition}

\begin{proposition}\label{prop:thetak}
We have
\begin{equation*}
\alpha(X)\leq \vartheta_k(X).
\end{equation*}
\end{proposition}

\begin{proof} 
Let $S$ be an independent set with $|S| =\alpha(X)$.
As $Y^S=\delta_{\binom{S}{k}}\delta_{\binom{S}{k}}^T$, the matrix $Y^S$  is clearly positive semidefinite.
We have
\begin{equation}
\langle Y^S,I\rangle = k\binom{|S|}{k}
\end{equation}
and 
\begin{equation}
\begin{array}{ll}
\langle Y^S,\Ld_{k-1}\rangle &= \displaystyle k^2\binom{|S|}{k} + \sum_{\substack{
    |F\cup F'|=k+1\\  F\cup F' \subseteq S}} 1\\
&=\displaystyle k^2\binom{|S|}{k} +(k+1)k\binom{|S|}{k+1} = \displaystyle k\binom{|S|}{k}|S|.
\end{array}
\end{equation}
Moreover, from the fact that $S$ is an independent set, and from the
definition of $Y^S$ \eqref{e26}, 
it is clear that
$(Y^S)_{F,F'}=0$ if $F\cup F'\in X_k$, or if $|F\cup F'|\geq
k+2$. 

The conditions $\epsilon_{F,F'} Y_{F,F'}=\epsilon_{F'',F^\dag}Y_{F'',F^\dag}$
if $F\cup  F'=F''\cup F^\dag$ are satisfied by the
entries of $\Ld_{k-1}$, so the matrix $Y^S$ inherits this property.

To sum up, we have proved that the matrix $k^{-1}\binom{|S|}{k}^{-1}Y^S$
  is feasible for $\vartheta_k(X)$.
Since its objective value is equal to $|S|$,  we can conclude that
$\alpha(X)\leq \vartheta_k(X)$.
\end{proof}

Now we consider the dual program of \eqref{eq:thetak primal}, in order
to obtain another formulation of $\vartheta_k(X)$, similar to
\eqref{eq:theta dual}. 

\begin{proposition}
We have 

\begin{equation}\label{eq:thetak dual}
\begin{array}{rl}
\vartheta_{k}(X)
=\inf \big\{\ \lambda_{\max}(Z) \ :\ &  Z=\Ld_{k-1} + T, \\
&T_{F,F}=0 \text{ for all }F\in \binom{V}{k}\\
&\sum_{ F\cup F'=H} \epsilon_{F,F'}T_{F,F'}=0 \ \text{ if
  } H\in \binom{V}{k+1}\setminus X_k\big\}
\end{array}
\end{equation}
\end{proposition}

\begin{proof} This is just a straightforward rewriting of the dual
  program. Both programs have the same objective value because
  Slater's condition holds: $Y=\binom{n}{k}^{-1}I$ is a strictly
  feasible solution of \eqref{eq:thetak primal} and $T=0$ gives rise to a  strictly feasible
  solution of \eqref{eq:thetak dual}.
\end{proof}

\begin{remark}\label{rk:thetak} Let us make a few obvious observations about
  $\vartheta_k(X)$. The first one, is that, as expected, $k\leq
  \vartheta_k(X)\leq n$. Indeed, the lower bound follows by taking
  $Y=\binom{n}{k}^{-1}I$ in \eqref{eq:thetak primal} while the upper
  bound follows by taking $T=0$ in \eqref{eq:thetak dual}.

The second observation is that $\vartheta_k(X)$ is easy to determine
for the empty and the complete $k$-complexes. Indeed, if $X$ is the
empty $k$-complex, the matrix $Y=k^{-1}\binom{n}{k}^{-1}\Ld_{k-1}$ is
feasible for \eqref{eq:thetak primal} giving that
$\vartheta_k(X)=n$. If $X$ is the complete $k$-complex, the
semidefinite program \eqref{eq:thetak primal} has only one feasible
solution which is $Y=\binom{n}{k}^{-1}I$ so $\vartheta_k(X)=k$. 

We note that, in these trivial cases, the equality $\alpha(X)=\vartheta_k(X)$ holds.
\end{remark}

The benefit of the formulation \eqref{eq:thetak dual} is that any
feasible matrix $T$ leads to an upper bound of $\vartheta_k(X)$ and
therefore to an upper bound of the independence number of $X$. Let us
illustrate this principle by showing that we can recover the upper bound proved by Golubev \cite{Go}
in the case of a $k$-dimensional simplicial complex $X$ with complete $(k-1)$-skeleton.

We take $T=\gamma(\Lu_{k-1}(X)-D_{k-1}(X))$ for some $\gamma\in \R$
that will be chosen later. Clearly $T$ satisfies the
conditions required by \eqref{eq:thetak dual}. Then
\begin{equation*}
\lambda_{\max}(\Ld_{k-1}+T)\leq \lambda_{\max} (\Ld_{k-1}+\gamma
\Lu_{k-1}(X))+\max_{F\in X_{k-1}}(-\gamma \deg(F)).
\end{equation*}

We assume that $X$ has complete $(k-1)$-skeleton, so we have $\Ld_{k-1}=\Ld_{k-1}(X)$
and $\Ld_{k-1}\Lu_{k-1}(X)=~0$. Let us denote by $\Lambda$ the set of
non zero eigenvalues of $\Lu_{k-1}(X)$. Then, the eigenvalues of the matrix $\Ld_{k-1}+\gamma \Lu_{k-1}(X)$ 
are:  $n$, associated to the eigenspace 
$B^{k-1}$, and  $\gamma\lambda$, for $\lambda\in \Lambda$, corresponding to
eigenvectors in $B_{k-1}$. For
$\gamma=\frac{n}{\lambda_{\max}(\Lu_{k-1}(X))}$, we have 
$\lambda_{\max} (\Ld_{k-1}+\gamma\Lu_{k-1}(X))=n$ and we get:
\begin{equation*}
\alpha(X)\leq \vartheta_k(X)\leq n\bigg(1-\frac{\deg_{\min}(X)}{\lambda_{\max}(\Lu_{k-1}(X))}\bigg).
\end{equation*}

We note that, if $X$ is regular, i.e., if $\deg(F)$ is a constant
number for $F\in \binom{V}{k}$, then this upper bound is the exact
analog of the ratio bound for graphs \eqref{eq:ratio bound}.

We have just seen that, in the case of a $k$-complex with complete $(k-1)$-skeleton, $\vartheta_k(X)$ is an upper bound of the independence
number of $X$ which is as least as good as the bound
\eqref{eq:bound Golubev}. The case of complexes with noncomplete $(k-1)$-skeleton turns out to be more tricky; indeed, in some cases
$\vartheta_k(X)$ provides a good bound of $\alpha(X)$, even a sharp
one, and beats the bound \eqref{eq:bound Golubev} given by Golubev,
while in other cases, Golubev's bound is better. 
We provide examples illustrating this situation in the next section, where we explicitly work out the computation of
$\vartheta_2(X)$ for certain families of $2$-dimensional complexes. This will also yield counterexamples for certain properties of the
theta number related to the chromatic number that we
might expect (see Section \ref{sec:chromatic}). It will also be interesting to observe the prominent
role plaed by the eigenvalues and eigenspaces of the Laplacian operators in these examples .

\section{The theta number of certain families of $2$-complexes}\label{sec:families}

\subsection{The complete tripartite $2$-complex}

To define this complex, we let $n=3m$ and partition $V=[n]$ into three subsets $A$, $B$,
$C$ of equal size $m$. As $2$-dimensional faces we select all triangles with exactly one vertex in each of these subsets; as $1$-dimensional faces all edges with at most one vertex in each of these subsets. 
A natural notation for this complex is $K_{m,m,m}^2$.
It is clear that
$\alpha(K_{m,m,m}^2)=2m$ because $A\cup B$ is a maximal independent set
with $2m$ vertices. We will show that $\vartheta_2(K_{m,m,m}^2)=2m$. 

With the notations of \eqref{eq:bound Golubev}, $d_0=2m$, $d_1=m$,
$\mu_0=3m$, $\mu_1=3m$ and the bound in \eqref{eq:bound Golubev}
equals $(7m-1)/3$, so this is an example where the theta number beats Golubev's
bound. 

We will also show that, for the complementary complex
$\overline{K_{m,m,m}^2}$, we have 
$\vartheta_2(\overline{K_{m,m,m}^2})=3=\alpha(\overline{K_{m,m,m}^2})$. 
This complex has a complete $1$-skeleton with $d_1=2m-2$ and
$\mu_1=3m$, so Golubev's bound \eqref{eq:bound Golubev} equals $(m+2)$, which
is not tight.

\begin{proposition}
We have $\vartheta_2(K_{m,m,m}^2)=2m$ and $\vartheta_2(\overline{K_{m,m,m}^2})=3$.
\end{proposition}

\begin{proof}
To keep notations light we use the generic notation $X$ for $X=K_{m,m,m}^2$ throughout the proof. 
We will verify that $\vartheta_2(X)=2m$, by
constructing a suitable matrix $T$ feasible for \eqref{eq:thetak
  dual}. The matrix $T$ will be constructed from the
projection matrices associated to certain eigenspaces of $\Lu_1(X)$
and $\Ld_1(X)$.

We denote by $A\times B$  the set of edges
connecting one vertex in $A$ and one vertex in $B$, and similarly for
the other kinds of edges. So, $X_1=(A\times B)\cup (B\times C)\cup
(C\times A)$. We choose the orientations of the triangular faces
and of the edges of $X$ following the rule $A\to B\to C\to A$; this
way, $[G:F]=+1$ for all $G\in X_2$ and $F\in X_1$. 

It turns out that the up-Laplacian $\Lu_1(X)$ has three non zero
eigenvalues,  $3m$, $2m$ and $m$, respectively with 
multiplicity $1$,  $3(m-1)$, and $3(m-1)^2$. We will need the projection matrices $P_{3m}^\uparrow$ and $P_{2m}^\uparrow$
associated to the eigenvalues $3m$ and $2m$.

The all-one
vector is clearly an eigenvector of $\Lu_1(X)$ for the eigenvalue
$3m$, so
$P_{3m}^{\uparrow}=J_{3m^2}/(3m^2)$.
The space $V_A=\{ \sum_{a\in A} x_a (\one_{a\times B}+\one_{a\times C})
\ :\ \sum_{a\in A} x_a=0\}$ is easily seen to be an eigenspace 
of $\Lu_1(X)$ for the eigenvalue $2m$. Similarly, we have two other
$(m-1)$-dimensional eigenspaces $V_B$ and $V_C$, and these spaces are
pairwise orthogonal. In order to express the projection matrix $P_{2m}^{\uparrow}$
associated to the sum of these spaces, we introduce the following notation: for $(F,F')\in X_1^2$, we denote $F\sim F'$ if
$F$ and $F'$ both belong to $A\times B$ (respectively to $B\times C$,
$C\times A$). Then,

\begin{equation*}
(P_{2m}^{\uparrow})_{F,F'}=\frac{1}{2m^2}\cdot
\begin{cases}
2(m-1)&\text{ if }  F=F' \\
-2 &\text{ if }  F\sim F' \text{ and }F\cap  F'=\emptyset  \\
(m-2) &\text{ if }  F\sim F' \text{  and }F\cap  F'\neq \emptyset , F\neq F'\\
-1 &\text{  if } F\not\sim F' \text{  and }F\cap F'=\emptyset\\
(m-1) &\text{  if } F\not\sim F' \text{  and }F\cap F'\neq \emptyset\\
\end{cases}
\end{equation*}

The down Laplacian $\Ld_1(X)$ has two non zero eigenvalues: $3m$ with
multiplicity $2$ and $2m$ with multiplicity $3(m-1)$. 
The vector space $\{ \gamma\one_{A\times B}+\alpha\one_{B\times C}+\beta\one_{A\times
  C} \ :\ \alpha+\beta+\gamma=0\}$ is a two-dimensional space of eigenvectors for
$\Ld_1(X)$ and for the eigenvalue $3m$, and the corresponding
projection matrix $P_{3m}^{\downarrow}$ is given by:
\begin{equation*}
(P_{3m}^{\downarrow})_{F,F'}=\frac{1}{3m^2}\cdot\begin{cases}
2 &\text{ if } F \sim  F' \\
-1 &\text{ otherwise.}
\end{cases}
\end{equation*}

So far the matrices that we have  defined are indexed by
$X_1=(A\times B)\cup (B\times C) \cup (A\times C)$. We now will
consider matrices indexed by the whole set $\binom{V}{2}$, therefore we
extend the matrices introduced above by adding zero rows and columns
for the indices not belonging to $X_1$ (we keep the same notation for
the enlarged matrices). We are now ready to define the matrix $T$ that
will do the job for $\vartheta_2(X)$:

\begin{lemma}\label{lemma:tripartite} With the previous notations, let 
\begin{equation*}
T=2m(P_{3m}^\uparrow+P_{2m}^\uparrow+P_{3m}^\downarrow)-\Ld_1(X).
\end{equation*}
This matrix satisfies the following properties:
\begin{enumerate}
\item $T_{F,F}=0$ for all $F\in \binom{V}{2}$
\item $T_{F,F'}=0$ for all $F,F'$ such that $F\cap F'\neq \emptyset$
  and $F\cup F'\notin X_2$
\item $2m \Id -\Ld_1
-T \succeq 0$.
\end{enumerate}
\end{lemma}

\begin{proof} Properties (1) and (2) follow by direct verification. In
  order to prove (3), we write  $\Ld_1+T = U+V+W$ where 
$U=2m(P_{3m}^\uparrow+P_{2m}^\uparrow)$, $V=2m
P_{3m}^\downarrow$ and $W=\Ld_1(X)-\Ld_1$, and make the remark that 
the product of any two of these matrices is zero. Indeed, for $U,V$ and for $U,W$
it follows immediately from the property that the
product of up and down Laplacians is zero; for $V,W$, it is due to the
fact that the image of $P_{3m}^\downarrow$ is an eigenspace for the eigenvalue $3m$ not only
for $\Ld_1(X)$ but also for $\Ld_1$.
So, we need to prove that $2m\Id-U$, $2m\Id-V$ and $2m\Id-W$ are
positive semidefinite. For the first two it is obvious because $2m\Id-U=2m(\Id-
P_{3m}^\uparrow-P_{2m}^\uparrow)$ and
$2m\Id-V=2m(\Id-P_{3m}^\downarrow)$. 
So now the only missing piece is a proof that $2m\Id
-(\Ld_1-\Ld_1(X))\succeq 0$. 

For this, we arrange the elements of $\binom{V}{2}$ so that  those in
$X_1=(A\times B)\cup (B\times C)\cup (C\times A)$ come before those in
$(A\times A)\cup (B\times B) \cup (C\times C)$, and we accordingly write $\Ld_1$ by blocks:
$$\Ld_1=\begin{pmatrix} \Ld_1(X) & M \\
  M^T & N \end{pmatrix}. $$ We want to prove that 
\begin{equation*}
\begin{pmatrix}
2m\Id & -M \\
-M^T & 2m\Id-N
\end{pmatrix} \succeq 0.
\end{equation*}
By the Schur complement lemma, this is equivalent to $2m\Id-N -(2m)^{-1}M^T
M\succeq 0$. A direct computation shows that $M^T M=2m N$, so 
all boils down to $m\Id-N\succeq 0$, which is indeed true because $N$ is
a block-diagonal matrix with three blocks equal to $\Ld_1(K_m^2)$.
\end{proof}

Now, we turn our attention to $\overline{K_{m,m,m}^2}=\overline{X}$. In
order to prove that $\vartheta_2(\overline{X})=3$, we will use the
primal formulation \eqref{eq:thetak primal} and apply a symmetry
argument.
In the next section we will see a second, simpler, proof, using chromatic numbers, see Example \ref{ex:chromnum}.

With the previous notations, a feasible matrix $Y$ must be of the form:
\begin{equation*}
Y=\begin{pmatrix}
Y_1 & 0 \\ 0 & \tau \Id 
\end{pmatrix}
\end{equation*}
where $Y_1$ is supported on the diagonal and on the triangles that
belong to $X_2$, i.e., the triangles with one vertex in each of $A$,
$B$, $C$. It is clear that the automorphism group of $X$ permutes
transitively the  elements of $X_2$ and of $X_1$, and that, by convexity, \eqref{eq:thetak primal}
has a symmetric solution. So, without loss of generality, we can assume
that $Y_1=\beta \Lu_1(X)+\gamma \Id$. Restricting the semidefinite
program on this  set of matrices leads to 
a linear program in the variables $\beta$, $\gamma$, $\tau$ that can
be easily solved and leads to the optimal value $3$. We skip the
details here.

We note that this approach would not work for $\vartheta_2(X)$ because
$\overline{X}_2$ has two orbits: the triangles that are fully
contained in one of the subsets $A$, $B$, $C$ and the ones that
have two vertices in one of these sets and one vertex in another one.
\end{proof}

\subsection{The complete bipartite $2$-complex}

Now $n=2m$ and $V=[n]$ is partitioned in two subsets $A$, $B$,
of equal size $m$. As $2$-dimensional faces we select the triangles that meet both sets $A$ and $B$, thus having two vertices in one
of the parts and the third vertex in the other.
We denote this complex by $K_{m,m}^2$.
It is clear that $\alpha(K_{m,m}^2)=m$ since $A$ is an independent set
with $m$ vertices. This complex has a complete $1$-skeleton and $d_1=m$, $\mu_1=2m$ so the bound  \eqref{eq:bound Golubev
  complete} equals $m$, showing that $\vartheta_2(K_{m,m}^2)=m$ and that the theta number
agrees with Golubev's bound. 

For the complementary complex
$\overline{K_{m,m}^2}$, which is nothing else than the disjoint union
of two complete complexes $K_m^2$, we have $\alpha(\overline{K_{m,m}^2})=4$. 
Golubev's bound  is twice
 the value corresponding to $K_m^2$, thus $4$, and it is sharp
 again. As we will see know, $\vartheta_2(\overline{K_{m,m}^2})$ is
 much larger:

\begin{proposition}\label{prop:bipartite}
We have $\vartheta_2(K_{m,m}^2)=m$ and $\vartheta_2(\overline{K_{m,m}^2})=\frac{8m-4}{m+1}$.
\end{proposition}

\begin{proof} We let $X=K_{m,m}^2$. 
To compute $\vartheta_2(\overline{X})$, we again apply the symmetry principle, like  in the case of the complement of the
  tripartite complex. The automorphism group of $K_{m,m}^2$  has two
  orbits in $X_1=\binom{V}{2}$: the set $X_1^{\text{in}}$ of edges
  contained in $A$ or in $B$, having degree $m$, and  the set  $X_1^{\text{out}}$
   of 'crossing' edges, with degree $2(m-1)$. It acts transitively on
   the $2$-faces. So without loss of generality  a feasible matrix $Y$ of the
  primal formulation of $\vartheta_2(\overline{X})$ can be
  assumed to be  
\begin{equation*}
Y=\beta\Lu_1(X) +\gamma \Id_{\text{out}} +\tau \Id_{\text{in}}
\end{equation*}
where $\Id_{\text{out}}$ and $\Id_{\text{in}}$ denote the $0-1$ diagonal
matrices associated to respectively $X_1^{\text{out}}$ and
$X_1^{\text{in}}$.
The expressions of $\langle I,Y \rangle$ and of $\langle
\Ld_1,Y\rangle$ are linear in the variables $\beta,\gamma,\tau$, but
the condition that $Y$ is positive semidefinite is slightly more
complicated because $\Lu_1(X)$ does not commute with
$\Id_{\text{out}}$ and $\Id_{\text{in}}$. In fact, this condition
leads to quadratic constraints, as it will become clear if we 
write the matrices by blocks according to   $\binom{V}{2}
  =X_1^{\text{in}}\cup X_1^{\text{out}}$. It is easy to verify that
\begin{equation*}
\Lu_1(X)=\begin{pmatrix} m\Id & -M\\-M^T & 2m\Id-N \end{pmatrix},\quad
M^TM=mN-2J
\end{equation*}
and that $N$ has two non zero eigenvalues: $2m$, with multiplicity $1$
and eigenvector the all-one vector, and $m$, with multiplicity  $2(m-1)$.
Then, by the Schur complement lemma, the condition 
\begin{equation*}
\beta\Lu_1(X) +\gamma \Id_{\text{out}} +\tau \Id_{\text{in}}= 
\begin{pmatrix} (m\beta+\tau)\Id & -\beta M\\
\beta M^T & (2m\beta +\gamma)\Id -\beta N
\end{pmatrix}\succeq 0
\end{equation*}
leads to quadratic inequalities. It is a bit technical but not
difficult to see that an optimal solution satisfies $\gamma=\tau$,
and finally that it is
\begin{equation*}
Y=\frac{-1}{m^2(m+1)} \Lu_1(X) +\frac{2}{m(m+1)}\Id,
\end{equation*}
leading to the optimal value $\langle \Ld_1,Y\rangle = (8m-4)/(m+1)$.
\end{proof}

\section{Chromatic numbers}\label{sec:chromatic}

Let us first review the case of graphs. For a graph $G$, the \emph{clique number}
$\omega(G)=\alpha(\overline{G})$ and the chromatic number $\chi(G)$
are related by the obvious inequality $\alpha(\overline{G})\leq
\chi(G)$, and the theta number $\vartheta(\overline{G})$ lies in between
these numbers (\cite[Lemma 3, Corollary 3]{Lov}):
\begin{equation}\label{eq1}
\alpha(\overline{G})\leq \vartheta(\overline{G})\leq \chi(G).
\end{equation}
Moreover, the inequality $\vartheta(\overline{G})\leq \chi(G)$ is always
at least as strong as the inequality
$n/\vartheta(G)\leq \chi(G)$; indeed, we know that
$n\leq \vartheta(G)\vartheta(\overline{G})$ from \cite[Corollary
2]{Lov}.

Let us consider the situation for pure $k$-dimensional simplicial
complexes.  By analogy with graphs, the chromatic number $\chi(X)$ of a complex $X$,
is usually defined to be the least number of colors needed to color the
vertices of $X$ such that no $k$-face is monochromatic. 
We remark that for the complete $k$-complex $K_n^k$, the
color classes of an admissible coloring  cannot have more than
$k$ elements,  and consequently that $\chi(K_n^k)=\lceil n/k
\rceil$. So, for all $k$-dimensional complexes $X$, we have
$\alpha(\overline{X})\leq k\chi(X)$. Given 
that we have defined a generalization of the theta number to
$k$-complexes, that satisfies $\alpha(\overline{X})\leq
\vartheta_k(\overline{X})$, it
is natural to wonder if the inequality 
\begin{equation}\label{eq1X}
\vartheta_k(\overline{X})\leq k\chi(X).
\end{equation}
is also satisfied. Unfortunately, this is not true in general.
Indeed, from
the results of Section \ref{sec:families},  one can see that \eqref{eq1X}
is satisfied for the complete tripartite complex and for its complement, but fails for the complete bipartite
complex $K_{m,m}^2$, for which
$\vartheta_2(\overline{K_{m,m}^2})=(8m-4)/(m+1)$ (Proposition
\ref{prop:bipartite}) while $\chi(K_{m,m}^2)=2$. 

Let us now see if we can modify the definition of
the chromatic number of a simplicial complex, so that it
fits better with our theta number. To achieve this, we will adapt the
concept of graph homomorphisms to simplicial complexes. Indeed, a nice way
to understand the notions of chromatic and clique numbers of
graphs is through their connection to graph homomorphisms, as we will recall now.

A homomorphism $f$ from a graph $G$ to a graph
$G'$ is a mapping from the vertices of $G$ to the vertices of $G'$
that sends an edge of $G$ to an edge of $G'$. Then, the clique number and the chromatic number have the following interpretations: the clique number
$\omega(G)$ is the largest number $\ell$ such that there is a
homomorphism from the complete graph $K_\ell$ to $G$, and similarly
 $\chi(G)$ is the smallest number $\ell$ such that there is a
homomorphism from $G$ to $K_\ell$. Moreover, one can prove that, if
there is a homomorphism from $G$ to $G'$, then
$\vartheta(\overline{G})\leq \vartheta(\overline{G'})$. The
combination of these properties immediately leads to \eqref{eq1}.

In order to follow a similar approach for simplicial complexes, 
we introduce an ad-hoc  notion of homomorphism. 


\begin{definition} Let $X$ and $X'$ be two pure $k$-dimensional  simplicial complexes. A homomorphism $f$ from $X$ to $X'$  is a 
mapping $f: X_{k-1}\to X'_{k-1}$ with the following property:
There exist orientations of $X$ and $X'$ such that for every $H\in X_k$, there is $H'\in X'_k$ such that
\begin{enumerate}
\item  $\{f(F)\ : F\in X_{k-1},\ F\subset H\}=\{F'\in X'_{k-1}\ :\ F'\subset H'\}$,
\item  $[H':f(F)]=[H:F]$ for all $F\in X_{k-1}$ with $F\subset H$.
\end{enumerate}
\end{definition}

We note that this definition coincides in dimension $1$  with the
usual notion of  a graph homomorphism as one can always find suitable orientations. 

\begin{remark} In this definition, it is important to understand that a
  homomorphism $f$ may not necessarily be induced by a global mapping $f_0$
  between the vertices, i.e., it may be the case that there is no
  mapping $f_0:X_0\to X'_0$ such that $f(F)=f_0(F)$ for all $F \in X_{k-1}$. As an example consider the $2$-dimensional complex $X$ depicted in Figure~\ref{FIGUREexample1}. 
  
  Furthermore, condition (2) is not automatically fulfilled. The $2$-dimensional complex $X$ depicted in Figure~\ref{FIGUREexample2} possesses a map $f: X_{1}\to (K_3^2)_{1}$ satisfying condition (1) but there is no homomorphism  from $X$ to $K_3^2$. 
\end{remark}

\begin{figure}[htbp]
  \centering\includegraphics[scale=.5]{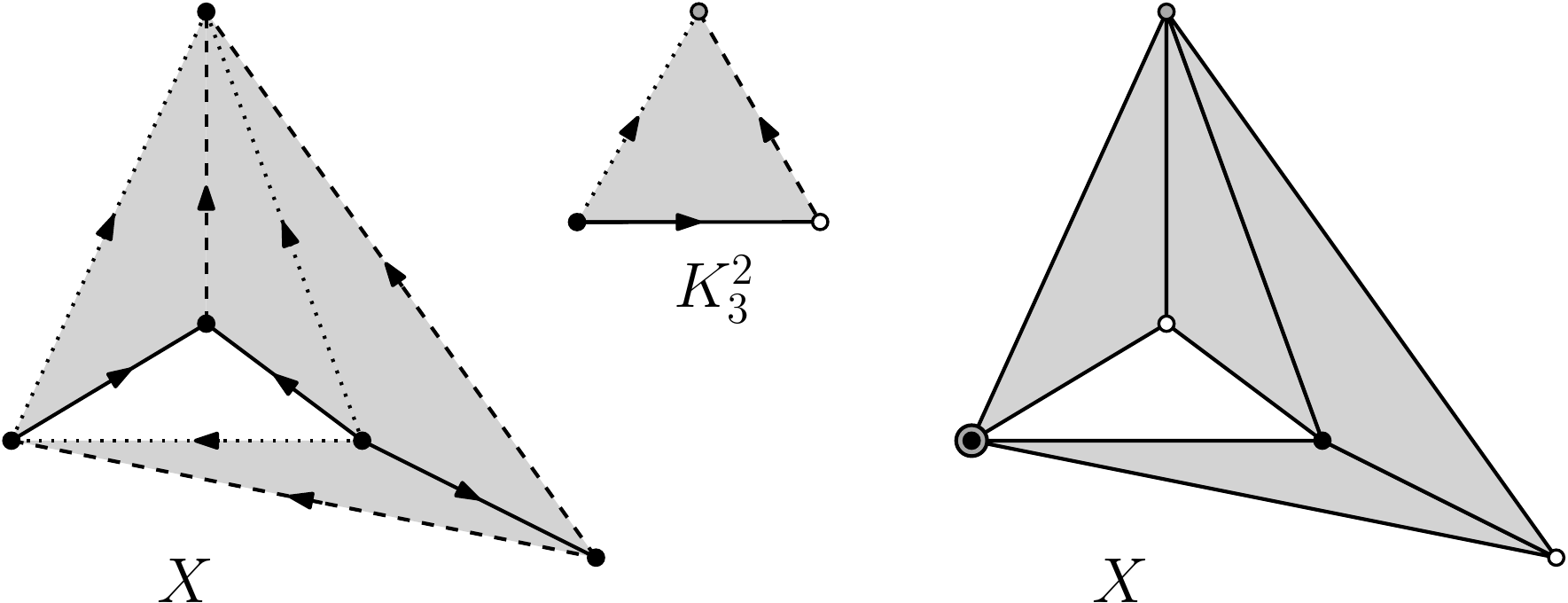}
  \caption{The homomorphism of $X$ to $K_3^2$ is not induced by a vertex map.}\label{FIGUREexample1}
\end{figure}

\begin{figure}[htbp]
  \centering\includegraphics[scale=.5]{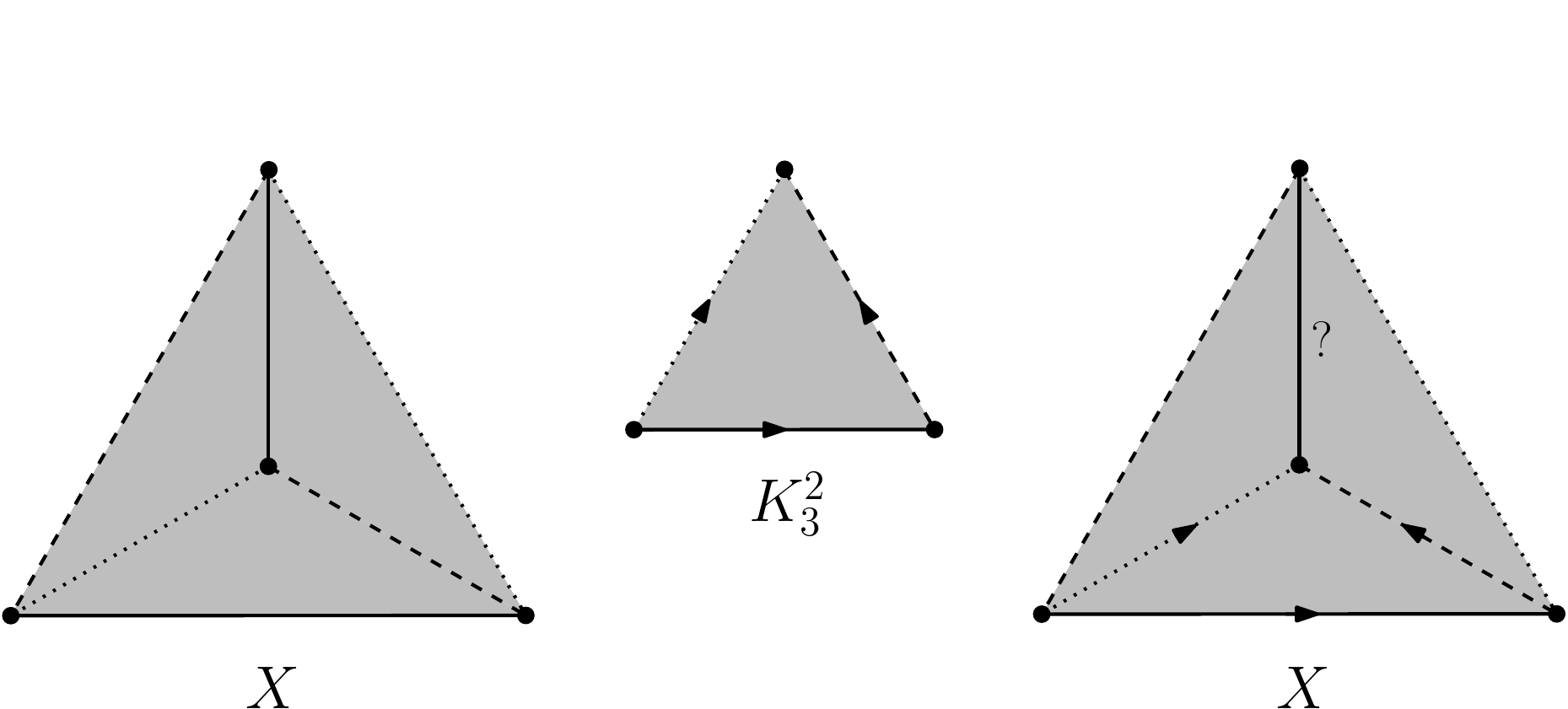}
  \caption{A complex $X$ with no homomorphism to $K_3^2$}\label{FIGUREexample2}
\end{figure}

\begin{proposition} Let $X$ and $X'$ be two pure $k$-dimensional
  simplicial complexes, and let $f$ be a homomorphism  from $X$ to
  $X'$.
Then, 
\begin{equation}\label{eq:theta hom}
\vartheta_k(\overline{X})\leq \vartheta_k(\overline{X'}).
\end{equation}
\end{proposition}

\begin{proof} Our strategy will be to start with an optimal solution $Y$ of the primal
  formulation \eqref{eq:thetak primal} of $\vartheta_k(\overline{X})$,
  from which  we construct a  matrix $Y'$, feasible for
  $\vartheta_k(\overline{X'})$, and having the same objective value as $Y$.

So, let $Y$ be primal optimal for the semidefinite program defining $\vartheta_k(\overline{X})$. 
We remark that, if $F\notin X_{k-1}$, then, for all $F'\neq F$, $F\cup F'\notin X_k$, and so $Y_{F,F'}=0$.
As a consequence, by the optimality of $Y$, we have $Y_{F,F}=0$. 

For $(K,K')\in X_{k-1}^2$, we set
\begin{equation*}
Y'_{K,K'}=\sum_{\substack{ (F,F')\in X_{k-1}^2\\f(F)=K ,\, f(F')=K'} }Y_{F,F'}
\end{equation*}
where the sum is zero if $K$ or $K'$ does not belong to the image of $f$. 

We have $\trace(Y')=\sum_{K\in \binom{V}{k}} Y'_{K,K}=\sum_{F\in X_{k-1}} Y_{F,F}=\trace(Y)$.

By the property 1) of homomorphisms, if $K\neq K'$ and $K\cup K'$ is not an element
of $X'_k$, and if $K=f(F)$ and  $K'=f(F')$, then $F\cup F'$ cannot
belong to $X_k$, and so $Y_{F,F'}=0$. So, we have that
$Y'_{K,K'}=0$. 

Thanks to property 2), if $K\cup
K'\in X'_k$ and $K\cup K'=K''\cup K^\dag$, the required condition that 
$\epsilon_{K,K'}Y'_{K,K'}=\epsilon_{K'',K\dag}Y'_{K'',K\dag}$ holds. So, we have proved that $Y'$ is primal feasible for $\vartheta_k(\overline{X'})$.

It remains to analyze the objective value $\langle \Ld_{k-1}, Y'\rangle$.
We have 
\begin{equation*}
\langle \Ld_{k-1}, Y'\rangle=k\trace(Y')+\sum_{K,K'\ :\ K\cup K'\in X'_k} \epsilon_{K,K'}Y'_{K,K'}.
\end{equation*}
But 
\begin{align*}
\sum_{\substack{K,K'\\ K\cup K'\in X'_k}} \epsilon_{K,K'}Y'_{K,K'}&=
\sum_{\substack{K,K'\\ K\cup K'\in X'_k}} \epsilon_{K,K'}\sum_{\substack{ (F,F')\in X_{k-1}^2\\f(F)=K ,\, f(F')=K'} } Y_{F,F'}\\
&=\sum_{\substack{(F,F')\in X_{k-1}^2\\ F\cup F'\in X_k}} \epsilon_{F,F'}Y_{F,F'}
\end{align*}
where in the last equality we ignore the terms corresponding to $F\cup F'\notin X_k$ because they are equal to zero, and  we apply the property 2). It follows that
$\langle \Ld_{k-1}, Y'\rangle=\langle \Ld_{k-1}, Y\rangle$. 
\end{proof}

\begin{definition}
Let $X$ be a pure $k$-dimensional simplicial complex. Let $\chi_k(X)$ denote the smallest number
$\ell$ such that there exists a homomorphism from $X$ to the complete $k$-complex $K_{\ell}^k$.
\end{definition}

It is not hard to see that $\chi_k(X) \leq \chi(X_1)$ holds for any pure simplicial complex $X$ as a vertex coloring with $\ell$ colors that is a proper graph coloring for $X_1$ gives rise to a homomorphism from $X$ to $K_\ell^k$. The complex $X$ depicted in Figure~\ref{FIGUREexample1} serves as an example that the three notions of chromatic numbers considered here differ. It has $\chi_2(X)=3$, $\chi(X) = 2$ and $\chi(X_1)=4$.

\begin{proposition}
We have
\begin{equation*}
\vartheta_k(\overline{X})\leq \chi_k(X).
\end{equation*}
\end{proposition}

\begin{proof}
If there is $f: X\to K_{\ell}^k$ then applying \eqref{eq:theta hom} leads to 
$\vartheta_k(\overline{X})\leq \vartheta_k(\overline{K_{\ell}^k})=\ell$ (see Remark \ref{rk:thetak}).
\end{proof}

\begin{example}\label{ex:chromnum}
Consider the complex $X=K_{m,m,m}^2$ defined in Section \ref{sec:families}. Clearly, $\chi_2(X)=\chi(X_1)=3$, so we have $3=\alpha(\overline{X}) \leq \vartheta_2(\overline{X})\leq \chi_2(X)=3$ and hence $\vartheta_2(\overline{X})=3$.
\end{example}

A $k$-dimensional subcomplex $C$ of a pure $k$-dimensional simplicial complex $X$ is a \emph{connected component} of $X$ if for every $(k-1)$-face $F$ of $C$ any $k$-face of $X$ that contains $F$ is also in $C$. Note that this condition does not need to hold for lower dimensional simplices, so two distinct connected components can, e.g., share a common vertex. Further observe that the connected components of $X$ correspond to the connected components of the graph that has the $k$-faces of $X$ as vertices with two vertices forming an edge if the correponding $k$-faces intersect in a common $(k-1)$-face.

As different connected components do not share $(k-1)$-faces, the inequality $\chi_k(X) \leq \chi(X_1)$ can actually be extended to the connected components of $X$.

\begin{proposition} Let $\mathcal{C}$ be the collection of connected components of $X$. Then
\[\chi_k(X) \leq \max_{C \in\mathcal{C}}\chi(C_1).\]
\end{proposition}

It is well-known that a $d$-regular graph $G$ has a bipartite connected component if and only if the largest eigenvalue of the Laplacian is $2d$. 
In \cite{HJ} Horak and Jost present a combinatorial criterion that can be considered as a higher-dimensional analog of this:
They show that for a $d$-regular $k$-complex $X$ the largest eigenvalue of the Laplacian $\Lu_{k-1}(X)$ is $(k+1)d$ if and only if there is a connected component $C$ of $X$ and an orientation of the $k$-faces of $X$ such that 
$[H:F] = [H':F]$ for all $F\in C_{k-1}$, $F \subset H,H'$.
Note that for a connected graph the existence of such an orientation is equivalent to bipartiteness.

If a $k$-dimensional simplicial complex $X$ has chromatic number $\chi_k(X)=k+1$, this guarantees the existence of such an orientation. Hence, we have the following observation.
\begin{proposition}
Let $X$ be a $d$-regular $k$-dimensional simplicial complex. If $\chi_k(X)=k+1$, then the maximal eigenvalue of the up-Laplacian $\Lu_{k-1}$ is $(k+1)d$.
\end{proposition}
We remark that these results extend to arbitrary complexes for a normalized version of the Laplacian that we do not study here.


\section{A hierarchy of semidefinite relaxations for the independence
  number of a $k$-simplicial complex}\label{sec:thetal}

In this section, $X$ is again a pure $k$-dimensional simplicial complex. 
We consider a straightforward generalization of 
$\vartheta_k(X)$ that leads to higher order theta numbers
$\vartheta_{\ell}(X)$ for $\ell>k$. We will see that all these numbers
provide upper bounds of $\alpha(X)$, until $\ell=\alpha(X)$, where
$\vartheta_{\alpha(X)}=\alpha(X)$.
Finally, we will modify this sequence of theta numbers in order to get a decreasing sequence. 

 It will be convenient to denote
by $\Ind_{i}$ the set of independent sets of dimension $i$. We make
the remark that $\Ind:=\Ind_{-1}\cup\dots\cup\Ind_{\alpha(X)-1}$ is a
simplicial complex, the \emph{independence complex} of $X$, and that it has complete $(k-1)$-skeleton,
i.e., $\Ind_{k-1}=\binom{V}{k}$. 
For $\ell>k$, the matrices involved in the program defining
$\vartheta_{\ell}(X)$ are indexed by $\Ind_{\ell-1}$.  We define, 
for $k\leq \ell \leq \alpha(X)$:

\begin{equation}\label{eq:thetal primal}
\begin{array}{rl}
\vartheta_{\ell}(X)=\sup\big\{ \langle \Ld_{\ell-1}(\Ind),Y\rangle \ :&Y\in
                                                             \R^{\Ind_{\ell-1}\times
                                                                         \Ind_{\ell-1}},\
                                                             Y\succeq
0, \ \langle I,Y\rangle=1, \\
&Y_{F,F'}=0 \text{ if } F\cup F'   \in \binom{V}{\ell+1}\setminus \Ind_{\ell},\\
&Y_{F,F'}=0 \text{ if }|F\cup F'|\geq \ell+2,\\
&\epsilon_{F,F'} Y_{F,F'}=\epsilon_{F'',F^\dag}Y_{F'',F^\dag} \text{ if } F\cup
  F'=F''\cup F^\dag\,\big\}
\end{array}
\end{equation}
and its dual formulation:
\begin{equation}\label{eq:thetal dual}
\begin{array}{rl}
\vartheta_{\ell}(X)=\inf \big\{\ \lambda_{\max}(Z) \ :\ &  Z=\Ld_{\ell-1}(\Ind)+ T, \\
&T_{F,F}=0 \text{ for all }F\in \Ind_{\ell-1},\\
&\sum_{ F\cup F'=H} \epsilon_{F,F'}T_{F,F'}=0 \ \text{ if
  } H\in \Ind_\ell\big\}
\end{array}
\end{equation}
The above definition matches for $\ell=k$ with that of
$\vartheta_k(X)$. Both primal and dual programs are strictly feasible:
$Y=I/\la I,I\ra $ and respectively $T=0$ give rise to strictly feasible solutions. We note
that, if $\ell=\alpha(X)$, the feasible matrices of the primal program
are diagonal matrices 
and hence $\vartheta_{\ell}(X)=\ell=\alpha(X)$. We have 
\begin{proposition}\label{prop:thetal}
\begin{equation*}
\alpha(X)\leq \vartheta_{\ell}(X).
\end{equation*}
\end{proposition}
\begin{proof} The same proof as the one of Proposition
  \ref{prop:thetak} works. For an independent set $S$ such that
  $|S|\geq \ell$, we define  $Y^S\in \R^{\Ind_{\ell-1}\times
    \Ind_{\ell-1}}$ by
 \begin{equation*}
(Y^S)_{F,F'} = \left\{\begin{array}{ll}
0 &\text { if } F \cup F' \nsubseteq S\\
(\Ld_{\ell-1}(\Ind))_{F,F'} &\text{ otherwise}.
\end{array}
\right.
\end{equation*}
It is then easy to verify,  as every subset of an independent set $S$ is also an independent
  set, that $\ell^{-1}\binom{|S|}{\ell}^{-1}Y^S $
is feasible for the primal program \eqref{eq:thetal primal} and that
its objective value is equal to $|S|$. 
\end{proof}

However, it is not clear that the sequence
$(\vartheta_{\ell}(X))_{k\leq \ell\leq \alpha(X)}$ is decreasing,
because the constraints on the $\ell$-sets involved in
$\vartheta_{\ell-1}(X)$ do not occur explicitly in
$\vartheta_\ell(X)$. We now define a  variant of
$\vartheta_\ell(X)$ that provides a decreasing sequence of upper
bounds of $\alpha(X)$.

To start with, we note that, if  a matrix $Y$ is feasible for
\eqref{eq:thetal primal}, then the value of $\epsilon_{F,F'}Y_{F,F'}$
for $(F,F')$ such that $|F\cup F'|=\ell+1$ only depends on $F\cup
F'$. So, we can associate to $Y$ a function $y\in \R^{\Ind_{\ell}}$
such that $\epsilon_{F,F'}Y_{F,F'}=y(H)$ if $H=F\cup F'$. If we extend $y$ to
$\Ind_{\ell-1}$ by $y(F):=Y_{F,F}$,  we see that $y$ encodes every nonzero entry of $Y$. Said differently, we have a one to one
correspondence between $\R^{\Ind_{\ell-1}\cup \Ind_{\ell}}$ and the
set 
\begin{equation*}
{\mathcal Y}_{\ell-1} =\big\{ Y\in \R^{\Ind_{\ell-1}\times \Ind_{\ell-1}} \ : 
\begin{array}{ll}
&Y_{F,F'}=0 \text{ if } F\cup F'   \in \binom{V}{\ell+1}\setminus \Ind_{\ell},\\
&Y_{F,F'}=0 \text{ if }|F\cup F'|\geq \ell+2,\\
&\epsilon_{F,F'} Y_{F,F'}=\epsilon_{H,H'}Y_{H,H'} \text{ if } F\cup F'=H\cup H'\
\end{array}
\big\}
\end{equation*}
We record for later use that, if $y \in
\R^{\Ind_{\ell-1}\cup \Ind_{\ell}}$ corresponds to $Y\in {\mathcal
  Y}$ as above, then
\begin{equation}\label{eq:y trace}
\la I, Y\ra= \sum_{F\in \Ind_{\ell-1}} y(F)
\end{equation}
and 
\begin{equation}\label{eq:y objective}
\la \Ld_{\ell-1}(\Ind),Y\ra=\ell\sum_{F\in \Ind_{\ell-1}} y(F)
+\ell(\ell+1)\sum_{H\in \Ind_{\ell} } y(H).
\end{equation}

Now, we introduce, for $\ell\geq 2$, a map $\tau_{\ell-1}: \YY_{\ell-1} \to
\YY_{\ell-2}$. It will be more convenient to define $\tau_{\ell-1}$ on
the corresponding functions $y\in \R^{\Ind_{\ell-1}\cup
  \Ind_{\ell}}$, in the following way: let
\begin{equation*}
\begin{array}{cccc}
\tau_{\ell-1} \ :& \R^{\Ind_{\ell-1}\cup \Ind_{\ell}} &\to
  &\R^{\Ind_{\ell-2}\cup \Ind_{\ell-1}}\\
 & y &\mapsto & \tau_{\ell-1}(y)=z
\end{array}
\end{equation*}
where
\begin{equation*}
\begin{cases}
z(K)=\frac{1}{\ell}\sum_{F\in \Ind_{\ell-1} \,:\, K\subset F} y(F)
\quad \text{if } K\in \Ind_{\ell-2}\\
 z(F)=\frac{1}{\ell(\ell-1)} y(F)+\frac{1}{\ell-1}\sum_{H\in \Ind_{\ell} \, : \, F\subset H} y(H)
\quad \text{if } F\in \Ind_{\ell-1}
\end{cases}
\end{equation*}

We are now in the position to define our strengthening of
$\vartheta_\ell(X)$: Let

\begin{equation}\label{eq:thetal tilde}
\begin{array}{rl}
\hat{\vartheta}_{\ell}(X)=\sup\big\{ \langle \Ld_{\ell-1}(\Ind),Y\rangle \ :&Y\in
                                                             \R^{\Ind_{\ell-1}\times
                                                                         \Ind_{\ell-1}},\
                                                             Y\succeq
0, \ \langle I,Y\rangle=1, \\
& \tau_i\circ\tau_{i+1}\circ\dots\circ \tau_{\ell-1}(Y)\succeq 0\ \text{
  for all } i=1,\dots,\ell-1,\\
&Y_{F,F'}=0 \text{ if } F\cup F'   \in \binom{V}{\ell+1}\setminus \Ind_{\ell},\\
&Y_{F,F'}=0 \text{ if }|F\cup F'|\geq \ell+2,\\
&\epsilon_{F,F'} Y_{F,F'}=\epsilon_{F'',F^\dag}Y_{F'',F^\dag} \text{ if } F\cup
  F'=F''\cup F^\dag\,\big\}.
\end{array}
\end{equation}

\begin{theorem} The numbers $\hat{\vartheta}_\ell(X)$, $k\leq \ell\leq
  \alpha(X)$, satisfy:
\begin{enumerate}
\item $\hat{\vartheta}_\ell(X)\leq  {\vartheta}_\ell(X)$
\item $\alpha(X)=\hat{\vartheta}_{\alpha(X)}(X)\leq   \hat{\vartheta}_{\alpha(X)-1}(X)\leq \dots \leq
  \hat{\vartheta}_k(X)$.
\end{enumerate}
\end{theorem}

\begin{proof} 
That $\hat{\vartheta}_\ell(X)\leq  {\vartheta}_\ell(X)$ is clear since we have only added constraints on $Y$ in the
definition of $\hat{\vartheta}_\ell(X)$.

Let $S$ be an independent set, with $|S|\geq \ell$. Let,
  like in the proof of Proposition \ref{prop:thetal}, $Y^S_{\ell-1} \in \R^{\Ind_{\ell-1}\times\Ind_{\ell-1}}$ be defined by:
\begin{equation}\label{def_Y^S}
(Y^S_{\ell-1})_{F,F'} = \begin{cases}
0 \text { if } F \cup F' \nsubseteq S\\
(\Ld_{\ell-1}(\Ind))_{F,F'} \text{ otherwise}.
\end{cases}
\end{equation}
The element $y^S_{\ell-1}\in \R^{\Ind_{\ell-1}\cup \Ind_{\ell}}$
corresponding to $Y^S_{\ell-1}$ is given by: $y^S_{\ell-1}(F)=\ell$ if $F\subset S$, 
$y^S_{\ell-1}(H)=1$ if $H\subset S$, and otherwise $y^S_{\ell-1}$ takes
the value $0$.  We will need the following lemma:

\begin{lemma}\label{lem:taul} We have
\begin{equation*}
\tau_{\ell-1}(y^S_{\ell-1})= \frac{|S|-\ell+1}{\ell-1} y^S_{\ell-2}
\end{equation*}
for $y^S_{\ell-1}$ as defined in \eqref{def_Y^S}.
\end{lemma}

\begin{proof} Let $z:=\tau_{\ell-1}(y^S_{\ell-1})$. Let $K\in
  \Ind_{\ell-2}$. Every subset of $S$ is independent so the number of
  $F\in \Ind_{\ell-1}$ such that $K\subset F\subset S$ is
  $|S|-\ell+1$. So,
\begin{equation*}
z(K)=\frac{1}{\ell}\sum_{F\in \Ind_{\ell-1} \,:\, K\subset F}
y^S_{\ell-1}(F)=|S|-\ell+1.
\end{equation*}
Now let $F\in \Ind_{\ell-1}$. It is clear that, if $F$ is not
contained in $S$, $z(F)=0$. If $F\subset S$,
\begin{align*}
z(F)&=\frac{1}{\ell(\ell-1)} \ell +\frac{1}{\ell-1}\sum_{H\in \Ind_{\ell}\,:\, F\subset
  H\subset S} 1\\
&=\frac{1}{\ell-1} +
  \frac{1}{\ell-1}(|S|-\ell)=\frac{|S|-\ell+1}{\ell-1}.
\end{align*}

\end{proof}

Lemma \ref{lem:taul} shows that $\tau_{\ell}(Y^S_{\ell-1})$ is
positive semidefinite, and so, iteratively, that $\tau_i\circ\tau_{i+1}\circ\dots\circ
\tau_{\ell-1}(Y^S_{\ell-1})$ is positive semidefinite for every $i\leq
\ell-1$. We conclude that
$Y^S_{\ell-1}$ (after a suitable rescaling) is feasible for
$\hat{\vartheta}_\ell(X)$, and consequently that $\alpha(X)\leq
\hat{\vartheta}_\ell(X)$. We have already remarked that $\vartheta_{\alpha(X)}=\alpha(X)$ so
also $\hat{\vartheta}_{\alpha(X)}=\alpha(X)$. 

It remains to prove that the sequence of $\hat{\vartheta}_{\ell}$
is decreasing. For this, we start from an optimal solution $Y$ of
$\hat{\vartheta}_{\ell}$, and we show that $Z:=\tau_{\ell-1}(Y)$ is feasible
for $\hat{\vartheta}_{\ell-1}$ and that $\la \Ld_{\ell-1}(\Ind),Y\ra
=\la \Ld_{\ell-2}(\Ind),Z\ra$. 

It is clear that $Z\in \YY_{\ell-2}$ and that $Z$ is positive
semidefinite, as well as $\tau_i\circ\tau_{i+1}\circ\dots\circ
\tau_{\ell-2}(Z)\succeq 0$ for all $i\leq \ell-2$. That $\la I,Z\ra
=1$ follows easily from \eqref{eq:y trace} and from the definition of
$\tau_{\ell-1}$. It remains to take care of the objective
value. Applying \eqref{eq:y objective},

\begin{align*}
&\la \Ld_{\ell-2}(\Ind),Z\ra = (\ell-1)\sum_{K\in \Ind_{\ell-2}} z(K)
                              +\ell(\ell-1)\sum_{F\in \Ind_{\ell-1}}
                              z(F)\\
&= (\ell-1) \sum_K \frac{1}{\ell} \sum_{F\,:\,K\subset F}
  y(F)+(\ell-1)\ell\sum_F \Big(\frac{1}{\ell(\ell-1)}y(F)+\frac{1}{\ell-1}\sum_{H\,:\,F\subset H} y(H)\Big)
\end{align*}
where in the sums we restrict to elements in $\Ind$. Taking account of
the fact that every subset of an independent set is also an
independent set, we obtain
\begin{equation*}
\la \Ld_{\ell-2}(\Ind),Z\ra=\ell\sum_{F\in \Ind_{\ell-1}} y(F)
                              +\ell(\ell+1)\sum_{H\in \Ind_{\ell}}y(H)=
\la \Ld_{\ell-1}(\Ind),Y\ra.
\end{equation*}
\end{proof}

\section{Theta numbers of random complexes}\label{sec:thetak random}

A random model $X^k(n,p)$ for simplicial complexes of arbitrary fixed dimension $k$ was introduced by Linial and Meshulam~\cite{LiMe} as a higher dimensional
analog of the Erd\"os-R\'enyi model $G(n,p)$ for random graphs. It has
vertex set $[n]=\{1,\dots,n\}$, complete $(k-1)$-skeleton, and each
element of $\binom{[n]}{k+1}$ is added as a $k$-dimensional face of $X^k(n,p)$
independently with probability $p$. Here $p=p(n)$ is a function of
$n$, and we let $q:=1-p$.  In this section we analyze the theta number
of $X^k(n,p)$ for 'dense' complexes, i.e., for $p$ in the range
$[c_0\log(n)/n, 1-c_0\log(n)/n]$. 

The study of the theta number of random graphs $G(n,p)$ was
initiated by Juh\'asz in \cite{Ju} who proved that, in the case of
constant probability $p$,  $\vartheta(G(n,p))=\Theta(\sqrt{nq/p})$ holds with probability
tending to $1$. In
subsequent works, the range of probabilities  for which Juh\'asz'
result holds was extended, until 
in \cite{CO}, Coja-Oghlan was able to cover $c_0/n\leq p\leq 1-c_0/n$ for some
sufficiently large constant $c_0$.

We will restrict ourselves to the range $c_0\log(n)/n\leq p\leq
1-c_0\log(n)/n$ because we will need the following estimates:
\begin{theorem}[\cite{FO, HKP}] Let $A$ denote the adjacency matrix of $G(n,p)$.
For every $c>0$ there exists $c_0>0, c'>0, c''>0$ such that, if  $c_0\log(n)/n\leq p\leq 1-c_0\log(n)/n$,
\begin{equation}\label{eq:concentration1}
\lambda_{\max}(pJ-A)\leq c'\sqrt{pq(n-1)}
\end{equation}
and 
\begin{equation}\label{eq:concentration2}
|\lambda_{\min}(A)|\leq c''\sqrt{pq(n-1)}.
\end{equation}
with probability at least equal to $1-n^{-c}$.
\end{theorem}
With the above, it is rather straightforward to obtain:
\begin{theorem}\label{th:theta random graph}
For every $c>0$ there exists $c_0>0, c_1>0, c_2>0$ such that, if $c_0\log(n)/n\leq p\leq 1-c_0\log(n)/n$, 
\begin{equation}
c_1\sqrt{(n-1)q/p} \leq \vartheta(G(n,p))\leq c_2\sqrt{(n-1)q/p}.
\end{equation}
with probability at least equal to $1-n^{-c}$.
\end{theorem}
Indeed, following the method of Juh\'asz, the upper bound is obtained via
the dual formulation for the theta number \eqref{eq:theta dual} and
the matrix $Z=J-A/p$, where $A$ is the adjacency matrix of $G(n,p)$,
while the lower bound follows from the choice $Y=Y'/\la I,Y'\ra$ in the
primal formulation \eqref{eq:theta primal}, where
$Y=\overline{A}-\lambda_{\min}(\overline{A})I$, $\overline{A}$ being
the adjacency matrix of the complementary graph of $G(n,p)$.

\subsection{The theta number of $X^k(n,p)$}
We will establish the following similar result for random simplicial complexes $X^k(n,p)$:
\begin{theorem}\label{th:theta random complex}
For every $k\geq 1$ and  $c>0$, there exists $c_0>0,c_1>0,c_2>0$ such that, if $c_0\log(n)/n\leq p\leq 1-c_0\log(n)/n$, 
\begin{equation*}
c_1\sqrt{(n-k)q/p} \leq \vartheta_k(X^k(n,p))\leq c_2\sqrt{(n-k)q/p}.
\end{equation*}
with probability at least equal to $1-n^{-c}$.
\end{theorem}

For comparison, the independence number of $X^k(n,p)$ is of the order $(\log(n^k p)/p)^{1/k}$ (see \cite{KS}).
In the range $c_0\log(n)/n\leq p\leq 1-c_0\log(n)/n$, the
eigenvalues of the adjacency matrix of $X^k(n,p)$ have been studied in
\cite{GW}.  We will closely follow the methods developed in \cite{GW}, in
particular the role played by the so-called \emph{links} of $X$, an idea going back to the work of Garland \cite{Ga}.
By definition, for a $k$-dimensional simplicial complex $X$ and a
$(k-2)$-face $K$ of $X$, \emph{the link $\lk_X(K)$} is the graph
with vertices  $\{v\in V\ : \ K\cup\{v\}\in X_{k-1}\}$, and  edges
$\{\{v,w\} \ :\ K\cup\{v,w\}\in X_k\}$. 
In view of the proof of Theorem \ref{th:theta random complex}, we will first establish a
relationship between the theta number of a simplicial complex and that
of its links.

\begin{proposition}\label{prop:link}
Let $X$ be a $k$-dimensional simplicial complex with complete
$(k-1)$-skeleton. Then
\begin{equation}\label{eq:link}
\vartheta_k(X)\leq k\max_{K\in X_{k-2}} \vartheta(\lk_X(K)).
\end{equation}
\end{proposition}

\begin{proof} Let $K\in X_{k-2}$. For a matrix $Y\in
  \R^{\binom{V}{k}\times \binom{V}{k}}$, we introduce its 
  \emph{localization at $K$} denoted $Y_K$ and defined by:
\begin{equation*}
(Y_K)_{F,F'}=\left\{
\begin{array}{ll}
Y_{F,F'}  & \text{ if }K\subset F \cap F'\\
0 &\text{ otherwise.}
\end{array}\right.
\end{equation*}
Let $\rho_K \in \R^{\binom{V}{k}\times \binom{V}{k}}$ denote the
diagonal matrix with $[F:K]$ as diagonal entries. Then we observe
that
\begin{equation}\label{eq:JK}
\Ld_{k-1}=\sum_{K\in X_{k-2}}  \rho_K J_K \rho_K.
\end{equation}
and that, if $Y_{F,F'}=0$ for all $(F,F')$ such that $|F\cup F'|\geq k+2$, 
\begin{equation}\label{eq:YK}
Y=\sum_{K\in X_{k-2}} Y_K - (k-1)\diag(Y).
\end{equation}

Now let $Y$ be an optimal solution of \eqref{eq:thetak primal}. Taking
account of \eqref{eq:JK} and \eqref{eq:YK},
\begin{align*}
\vartheta_k(X)=\la \Ld_{k-1},Y\ra &=\la \sum_{K}
                                    \rho_K J_K \rho_K,
                                    \sum_{K} Y_K \ra -(k-1)\la
                                    \Ld_{k-1}, \diag(Y)\ra \\
&=\sum_{K,K'} \la \rho_K J_K \rho_K, Y_{K'}\ra -k(k-1).
\end{align*}
If $K\neq K'$, we have
\begin{equation*}
\la \rho_K J_K \rho_K, Y_{K'}\ra=
\left\{
\begin{array}{ll}
Y_{F,F} &\text{ if }K\cup K'=F\\
0 &\text{ otherwise}
\end{array}
\right.
\end{equation*}
so, since $\trace(Y)=1$,
\begin{equation*}
\vartheta_k(X)=\sum_K \la \rho_K J_K\rho_K,  Y_K \ra =\sum_K \la J_K, \rho_K
Y_K \rho_K\ra.
\end{equation*}
Now, the crucial observation is that the matrix $\rho_K Y_K \rho_K$
gives rise to a feasible matrix of the semidefinite program
\eqref{eq:theta primal} defining the theta number of
$\lk_X(K)$. Indeed, let $Z_K$ be the matrix indexed by $V\setminus K$ and
defined by $(Z_K)_{v,w}=(\rho_K Y_K \rho_K)_{K\cup\{v\},
  K\cup\{w\}}$. This matrix inherits some properties of $Y$: The matrix $Z_K$ is positive semidefinite,
the entries of $Z_K$ associated to edges of $\lk_X(K)$
are equal to $0$. With obvious notations, we have $\la J_K, \rho_K Y_K
\rho_K\ra=\la J,Z_K\ra$ and $\la I,Z_K\ra=\la I, Y_K\ra $ so
we obtain
\begin{equation*}
\vartheta_k(X)\leq \sum_K \la I, Y_K \ra \vartheta(\lk_X(K)).
\end{equation*}
We have $\sum_K \la I, Y_K\ra=k\la I,Y\ra=k$ so the announced inequality follows immediately.
\end{proof}

\begin{proof}[Proof of Theorem \ref{th:theta random complex}]

For the upper bound, we apply Proposition \ref{prop:link}. 
The link $\lk_X(K)$ of a ($k-2$)-face $K$ in a random complex $X=X^k(n,p)$ is an
Erd\"os-Renyi random graph
on $V\setminus K$ with the same probability $p$. We can thus apply 
Theorem \eqref{th:theta random graph} and a union bound to obtain
the result. We note that, since the number of such faces is of the order of $n^{k-1}$, for the probability of the bad event to be, say, less than $n^{-c}$ we need to apply Theorem \eqref{th:theta random
  graph} for the larger value $c+k-1$ instead of $c$, explaining the
need for an arbitrary large power of $n$ in the convergence speed of probabilities.

In order to find a lower bound of $\vartheta_k(X)$, we consider the
matrix $Y=\overline{A}-\lambda_{\min}(\overline{A})I$ where
$\overline{A}$ denotes the adjacency matrix of the complementary
$k$-complex $\OX$. The feasibility conditions of
\eqref{eq:thetak primal} are fulfilled by $Y$ except for the
normalization condition $\la I, Y\ra=1$.  We have $\langle I,Y\rangle= - \binom{n}{k} \lambda_{\min}(\overline{A})$.
Moreover, $\langle \Ld_{k-1},Y\rangle=
k(k+1)|\OX_k|-k\binom{n}{k}\lambda_{\min}(\overline{A})$, so
\begin{equation*}
\vartheta_k(X)\geq k\bigg(1 +\frac{(k+1) |\OX_k|}{-\binom{n}{k} \lambda_{\min}(\overline{A})}\bigg).
\end{equation*}
The number $|\OX_k|$ of $k$-faces of $\OX=X^k(n,q)$ is a random variable binomially distributed 
in $[\binom{n}{k+1}]$ with probability $q$. Hence, by a straightforward
application of a Chernoff bound, for every $c>0$,  $|\OX_k|$ is at least of the order
$\binom{n}{k+1} q$ with probability at least $1-n^{-c}$. 
It remains to upper bound $|\lambda_{\min}(\overline{A})|$. For this,
we apply the localization procedure that we have already encountered
in the proof of Proposition \ref{prop:link}:
\begin{equation*}
\overline{A}=\sum_{K\in X_{k-2}} \overline{A}_K.
\end{equation*}
Then, for every $x=(x_F)_{F\in \binom{V}{k}}$, if $x_K$ denotes the
vector obtained from $x$ by setting to $0$ the coordinates of $x$ associated to
faces $F$ not containing $K$, 
\begin{equation*}
\la \overline{A}x,x\ra=\sum_K\la \overline{A}_K x,x\ra=\sum_K\la
\overline{A}_K x_K,x_K\ra.
\end{equation*}
The matrix $\overline{A}_K$ has the same spectrum as
$\rho_K\overline{A}_K\rho_K$. The latter is identical to the adjacency
matrix $A_{\lk_{\OX}(K)}$ of the graph
$\lk_{\OX}(K)$ on the entries indexed by $\{ F=K\cup\{v\},\
v\in V\setminus K\}$, and zero elsewhere. So, its non-zero spectrum is
that of $A_{\lk_{\OX}(K)}$ and hence:
\begin{equation*}
\la \overline{A}x,x\ra\geq \sum_K 
\lambda_{\min}(A_{\lk_{\OX}(K)})\la x_K,x_K\ra.
\end{equation*}
The links $\lk_{\OX}(K)$ are random graphs $G(n-k+1,q)$ so,
applying \eqref{eq:concentration2} and a union bound, we find that,
with probability at least equal to $1-n^{-c}$, for a large enough
constant $c''$,
\begin{equation*}
\la \overline{A}x,x\ra\geq -c''\sqrt{pq(n-k)}\sum_K 
\la x_K,x_K\ra=-c''k\sqrt{pq(n-k)}\la x,x\ra.
\end{equation*}
We have obtained the desired upper bound
$|\lambda_{\min}(\overline{A})|\leq c'''\sqrt{pq(n-k)}$. Putting
everything together, we obtain the announced lower bound for $\vartheta_k(X)$.
\end{proof}

\subsection{The hierarchy of theta numbers of $G(n,p)$}\label{sec:random thetal}

In this last subsection, we restrict ourselves to the case of random graphs $G(n,p)$  
and analyze the hierarchy of theta numbers $\vartheta_{\ell}(G(n,p))$ for constant values of $\ell$. The restriction to random graphs, i.e., random complexes of dimension $1$, is purely for simplicity. The assumption of constant $\ell$, however, is essential.
Analyzing the complete hierarchy $\hat{\vartheta}_{\ell}(X)$ of a random complex $X$ for non-constant $\ell$ appears to be a difficult task. It would be interesting to know for which values of $\ell$ the theta number $\vartheta_{\ell}(G(n,p))$ is close to the independence number. Unfortunately, such questions seem to be out of the reach of the methods we apply here.

\begin{theorem}\label{th:thetal random graph}
For every $\ell\geq 1$ and  $c>0$, there exists $c_0>0,c_1>0,c_2>0$
such that, if $q^\ell\geq c_0\log(n)/n$ and $pq^{\ell-1}\geq c_0\log(n)/n$, 
\begin{equation*}
c_1\sqrt{nq^\ell/p} \leq \vartheta_\ell(G(n,p))\leq c_2\sqrt{nq^\ell/p}.
\end{equation*}
with probability at least equal to $1-n^{-c}$.
\end{theorem}

\begin{proof} We will sometimes use the expression \emph{with high probability}
  for an inequality that holds with probability at least
  $1-n^{-c}$ for all $c>0$, with appropriate constants depending on $c$. 

For an upper bound of $\vartheta_{\ell}(G(n,p))$, we apply
\begin{equation*}
\vartheta_{\ell}(G)\leq \ell \max_{K\in \binom{V}{\ell-1}}
\vartheta(\lk_G(K)).
\end{equation*}
Here, $\lk_G(K)$ is the graph on $V_K:=\{v\in V\ :
\{v,k\} \notin E(G(n,p)) \text{ for all } k\in K\}$ with edges $\{v,w\}$ if
  $K\cup\{v,w\}\in \binom{V}{\ell+1}\setminus \Ind_{\ell}$. 
If $K$ is independent, this condition simply means that $\{v,w\}$ is an edge of $G$, so
$\lk_G(K)$ is the graph $G[V_K]$ induced by $G$ on $V_K$. If
$G=G(n,p)$, the number
of vertices $n_K=|V_K|$ is itself a random variable. Since
$|K|=\ell-1$, $n_K$ follows a binomial distribution with parameters
$(n-\ell+1)$ and $q^{\ell-1}$. For $n_K$ to be concentrated around its
expected value $q^{\ell-1}(n-\ell+1)$ we need $q^{\ell-1}\geq
c_0\log(n)/n$ for some $c_0>0$.

Assuming  $n_K\leq cq^{\ell-1}n$ for some $c>0$, we have 
\begin{equation*}
\vartheta(G[V_K])\leq \vartheta(G(cq^{\ell-1}n,p))
\end{equation*}
because $G[V_K]$ can be viewed as an induced subgraph of $G(cq^{\ell-1}n,p)$.
We would like to apply Theorem \ref{th:theta random graph}. It requires $p$ and
$q$ to be greater that $c'_0 \log(q^{\ell-1}n)/(q^{\ell-1}n)$ and
holds with probability at least $1-(q^{\ell-1}n)^c$. 
All this will be fine if we assume:
\begin{equation*}
pq^{\ell-1}\geq c_1\log(n)/n\text{ and } q^{\ell}\geq c_1\log(n)/n
\end{equation*}
for a sufficiently large $c_1$.
With a union bound we obtain with high probability:
\begin{equation*}
\vartheta_{\ell}(G)\leq c\sqrt{nq^{\ell}/p}.
\end{equation*}

For the lower bound, we consider the matrix $Y=A-\lambda_{\min}(A)I$
where $A$ is the adjacency matrix of the $\ell$-skeleton of $\Ind$ and
we apply \eqref{eq:thetal primal}. We obtain
\begin{equation*}
\vartheta_{\ell}(X)\geq \frac{\la \Ld_{\ell-1}(\Ind),Y\ra}{\la I,Y\ra}
= \ell\Big( 1
+\frac{(\ell+1) |\Ind_{\ell}|}{-\lambda_{\min}(A)|\Ind_{\ell-1}|}\Big).
\end{equation*}
In order to estimate $|\lambda_{\min}(A)|$ we use $A=\sum_{K\in
  \Ind_{\ell-2}}A_K$
and remark that $A_K$ has the same non-zero eigenvalues as the adjacency matrix of the graph
$\lk_{\Ind}(K)$, itself being the graph $\overline{G}[V_K]$ induced by
$\overline{G}$ on $V_K$. 
We have 
\begin{align*}
\la Ax,x\ra&=\sum_{K\in \Ind_{\ell-2}}\la A_K x,x\ra=\sum_{K\in
             \Ind_{\ell-2}}\la A_K x_K,x_K\ra\\
&\geq \sum_{K\in \Ind_{\ell-2}} \lambda_{\min}(A_K) \la x_K,x_K\ra\\
&\geq \min_{K\in \Ind_{\ell-2}} \lambda_{\min}(A_K) \sum_{K\in
  \Ind_{\ell-2}}\la x_K,x_K\ra\\
&\geq \min_{K\in \Ind_{\ell-2}} \lambda_{\min}(A_K) \ell \la x,x\ra,
\end{align*}
so
\begin{equation*}
-\lambda_{\min}(A) =|\lambda_{\min}(A)|\leq \ell \cdot \max_K
|\lambda_{\min}(\overline{G}[V_K])|.
\end{equation*}

Like for the upper bound we have with high probability $n_K\leq
cq^{\ell-1}n$ for some $c>0$ and thus  
\begin{equation*}
|\lambda_{\min}(\overline{G}[V_K])|\leq
|\lambda_{\min}(G(cq^{\ell-1}n,q))|\leq c'\sqrt{pq^{\ell}n}
\end{equation*}
for some $c'>0$, under the same conditions on $p$ and $q$.

It remains to deal with the ratio $|\Ind_{\ell}|/|\Ind_{\ell-1}|$. For
this we will argue that $\Ind$ is almost regular. To be more precise
we apply double counting to the set
\begin{equation*}
D=\{(A,B)\in \Ind_{\ell-1}\times \Ind_\ell \ :\ A\subset B\}.
\end{equation*}
The number of $\ell$-subsets of $B$ is $\ell+1$ so
$|D|=(\ell+1)|\Ind_\ell|$.
For  a given $A$, the number $X_A$ of $B$ containing $A$ follows a
binomial distribution with parameters $n-\ell$ and $q^\ell$, with
expected value $q^\ell(n-\ell)$. With high probability (requires
$q^\ell\geq c\log(n)/n$) $X_A$ is larger that $c'q^\ell(n-\ell)$ and
so
\begin{equation*}
\frac{|\Ind_\ell|}{|\Ind_{\ell-1}|}\geq
\frac{c'q^\ell(n-\ell)}{\ell+1}.
\end{equation*}
Putting everything together and applying another union bound we obtain 
\begin{equation*}
\vartheta_\ell(G)\geq c\sqrt{nq^\ell/p}.
\end{equation*}

\end{proof}

\end{document}